\documentclass[11pt]{article}
\usepackage[utf8]{inputenc}
\usepackage[normalem]{ulem}
\usepackage{todonotes,setspace}
\usepackage{float}
\usepackage{comment}
\usepackage{amsmath,amsthm,amsfonts,color}
\usepackage{subcaption}
\usepackage{amssymb}
\usepackage{graphicx}
\usepackage{setspace}
\usepackage{mathrsfs}
\usepackage{graphicx}
\usepackage[table]{xcolor}
\usepackage{colortbl}
\usepackage[shortlabels]{enumitem}
\setlist[itemize]{topsep=3pt, partopsep=6pt, itemsep=1pt, parsep=1pt}

\usepackage{calc,cite}
\usepackage{tikz}
\usepackage{geometry}
\usepackage{pgffor}
\usepackage{ifthen}
\usetikzlibrary{arrows.meta}
\usetikzlibrary{decorations.markings}
\tikzstyle{vertex}=[circle, draw, inner sep=1.2pt, minimum size=3pt]
\tikzstyle{filledvertex}=[circle, draw, fill, inner sep=1.5pt, minimum size=3pt]

\newcommand{\modthree}[1]{\mathrm{(}#1 \mathrm{\ mod\  3)}}

\newtheorem{theorem}{Theorem}[section]
\newtheorem{conjecture}[theorem]{Conjecture}
\newtheorem{lemma}[theorem]{Lemma}

\newtheorem{fact}[theorem]{Fact}
\newtheorem{case}{Case}

\newtheorem{subcase}{Subcase}[case]

\newtheorem{claim}{Claim}
\newtheorem{subclaim}{Subclaim}[claim]

\counterwithin*{case}{section}

\begin{document}
\begin{spacing}{1.15}

\title{On 2-connected graphs without cycles of length 1 modulo 3}

\author{Yandong Bai, Hojin Chu, Binlong Li, Boram Park, Homoon Ryu}

\author{Yandong Bai
\thanks{School of Mathematics and Statistics, 
Xi’an--Budapest Joint Research Center for Combinatorics, Northwestern Polytechnical University, Xi'an 710129, China;
Research $\&$ Development Institute of Northwestern Polytechnical University in Shenzhen, Shenzhen 518057, China. Email: {\tt bai@nwpu.edu.cn}.}
\and Hojin Chu
\thanks{School of Computational Sciences, Korea Institute for Advanced Study, Seoul 02455, Republic of Korea. Email: {\tt hojinchu@kias.re.kr}.}
\and Binlong Li
\thanks{School of Mathematics and Statistics, Xi’an--Budapest Joint Research Center for Combinatorics, Northwestern Polytechnical University,
Xi'an 710129, China. Email: {\tt binlongli@nwpu.edu.cn}.}
\and Boram Park
\thanks{Department of Mathematics Education, Seoul National University, Seoul 08826, Republic of Korea. Email: {\tt borampark@snu.ac.kr}.}
\and Homoon Ryu
\thanks{Department of Mathematics Education, Seoul National University, Seoul 08826, Republic of Korea. Email: {\tt rhm95@snu.ac.kr}.}
}

\date{}

\maketitle

\begin{abstract}
Burr and Erd\H{o}s conjectured in 1976 that for all integers $k>\ell\geq 0$ such that $k\mathbb{Z}+\ell$ contains an even integer, every $n$-vertex graph without cycles of length $\ell$ modulo $k$ has at most a linear number of edges in $n$. Bollob\'{a}s confirmed the conjecture in 1977, and Erd\H{o}s further asked for the exact extremal number. To the best of our knowledge, this problem has been solved only for all residues when $k\leq 4$, and for $\ell\in \{0,2\}$ when $k\geq 5$ is odd. In particular, Bai {\it et al.} [arXiv:2503.03504] proved that if $G$ is an $n$-vertex graph with no cycles of length $1$ modulo $3$, then $e(G)\le \frac{5}{3}(n-1)$, and when $9\mid (n-1)$ the equality holds if and only if each block of $G$ is isomorphic to the Petersen graph. Note that for $n> 18$ every extremal graph contains a cut-vertex. In this paper, we investigate the 2-connected setting and determine the maximum number of edges in a 2-connected graph with no cycles of length $1$ modulo $3$. Our results provide a sharp extremal bound and a complete characterization of the extremal graphs, revealing structural differences from the general case. Combining this with the result of Bai {\it et al.}, we also obtain a complete characterization of all extremal graphs in the general setting, including the cases where $9\nmid (n-1)$. Finally, we determine the maximum number of edges in a $2$-connected graph with no cycles of length $2$ modulo $4$, whose extremal graphs differ substantially from those in the general setting. Consequently, the extremal numbers for $2$-connected graphs with no cycle of a fixed length modulo $k$ are now determined for all $k\leq 4$.
\end{abstract}

\section{Introduction}

For a graph $G$, we denote the sizes of $V(G)$ and $E(G)$ by $n(G)$ and $e(G)$, respectively. All graphs are simple in this paper. By an ($\ell$ mod $k$)-cycle (resp. ($\ell$ mod $k$)-path), we mean a cycle (resp. path) of length $\ell$ modulo $k$.  Let $\mathcal{C}_{\ell\bmod k}$ be the set of all $(\ell\bmod k)$-cycles.
For a family $\mathcal{F}$ of graphs, we denote by $ex(n,\mathcal{F})$ the maximum number of edges in an $n$-vertex graph containing no member of $\mathcal{F}$ as a subgraph. 
Determining $ex(n,\mathcal{C}_{\ell\bmod k})$ has attracted extensive attention in extremal graph theory.

Burr and Erd\H{o}s \cite{BurrErdos} conjectured in 1976 that such an $n$-vertex graph can contain at most a linear number of edges in $n$. Bollob\'{a}s \cite{B77} gave a positive answer to this conjecture. Erd\H{o}s then asked what the exact value of the maximum number of edges in such a graph is.

Let $G$ be an $n$-vertex graph with no $(\ell \bmod k)$-cycles for some integers $k>\ell\geq 0$ such that $k\mathbb{Z}+\ell$ contains an even integer. Then $k$ is odd or $\ell$ is even. The cases $k\leq 2$ are somewhat trivial. 
If $(k, \ell)=(1,0)$, or equivalently, $G$ has no cycles, then one can see that $e(G)\leq n-1$ and the equality holds if and only if $G$ is a tree. A famous result states that if $(k, \ell)=(2,0)$, then $e(G)\leq \frac{3}{2} (n-1)$ and, for $2\mid (n-1)$, the equality holds if and only if each block of $G$ is a triangle. 

The first non-trivial case is $k=3$.
Chen and Saito~\cite{CS94} showed that if $(k, \ell)=(3,0)$ then $e(G)\leq 2(n-2)$ for $n\geq 3$ 
and the equality holds if and only if $G$ is isomorphic to $K_{2,n-2}$. Dean {\it et al.} \cite{D+91}
and, independently, Saito \cite{Sai92} showed that for $n\geq 5$ if $(k, \ell)=(3,2)$, then $e(G)\leq 3(n-3)$ and the equality holds if and only if $G$ is isomorphic to $K_{3,n-3}$. 
Bai {\it et al.} \cite{B+202503} showed that if $(k, \ell)=(3,1)$, then $e(G)\leq \frac{5}{3} (n-1)$ and, for $9\mid (n-1)$, the equality holds if and only if each block of $G$ is a Petersen graph.

\begin{theorem}[Bai {\it et al.} \cite{B+202503}]\label{thm:1mod3}
For $n\geq 1$,
\[
ex(n,\mathcal{C}_{1\bmod 3})
=15q+\left\lfloor \frac{3}{2}r \right\rfloor,~
\text{where~}n-1=9q+r,~0\leq r\leq 8.
\]
Moreover, if $9 \mid (n-1)$, then the extremal graphs are those in which each block is a Petersen graph.
\end{theorem}

For the case $k=4$, since $k\mathbb{Z}+\ell$ contains an even integer, it suffices to consider $\ell=0$ or $\ell=2$. Gy\H{o}ri {\it et al.} \cite{G+23} showed that if $(k, \ell)=(4,0)$, then $e(G)\leq \frac{19}{12} (n-1)$ and they also constructed a family of graphs attaining this bound. Gao {\it et al.} \cite{G+24} showed that if $(k, \ell)=(4,2)$, then $e(G)\leq \frac{5}{2} (n-1)$ and, for $4\mid (n-1)$, the equality holds if and only if each block of $G$ is $K_5$.


For general $k\geq 3$, Bai {\it et al.} \cite{B+202511} showed that $e(G)\leq (k-1)(n-k+1)$ if $\ell=0$, and the bound is sharp when $k$ is odd and $n\geq 2k-3$. Their method also implies that $e(G)\leq k(n-k)$ if $\ell=2$, and the bound is sharp when $k$ is odd and $n\geq 2k$. It is mentioned that the latter result can also be obtained from the results of Gao {\it et al.} \cite{G+22}.

Let $ex_2(n,\mathcal{F})$ be the maximum number of edges in a $2$-connected $n$-vertex graph with no graphs in $\mathcal{F}$. It is not difficult to see that if a 2-connected graph contains no even cycles, then it is an odd cycle. So, $ex_2(n, \mathcal{C}_{0 \bmod 2})=n$ when $n\ge 3$ is odd and it is meaningless when $n$ is even. By examining the characterization of extremal graphs for $ex(n,\mathcal{C}_{\ell\bmod k})$ when $k\in \{3,4\}$, one can observe that $ex_2(n,\mathcal{C}_{\ell\bmod k}) = ex(n,\mathcal{C}_{\ell\bmod k})$ for $(\ell,k)\in \{(0,3),(2,3)\}$ for large $n$, whereas the two parameters differ for $(\ell,k)\in \{(1,3),(0,4),(2,4)\}$.
Very recently, Chu, Park, and Ryu \cite{CPR2025} determined $ex_2(n,\mathcal{C}_{0\bmod 4})$.

In this paper, we determine both $ex_2(n,\mathcal{C}_{1\bmod 3})$ and $ex_2(n,\mathcal{C}_{2\bmod 4})$, thereby completing the determination of the extremal number for $2$-connected graphs with no cycles of length congruent to a fixed residue modulo $k$ for all $k\le 4$.

It is not difficult to deduce that $ex_2(n,\mathcal{C}_{2\bmod 4}) = 2n-2$, and that the extremal graphs are obtained from $K_{2,n-2}$ by adding one edge in each part for $n\ge 6$ (see Section~\ref{sec:concluding}).
We therefore focus primarily on determining $ex_2(n,\mathcal{C}_{1\bmod 3})$. Note that the value $ex_2(4,\mathcal{C}_{1\bmod 3})$ is not defined as any $2$-connected $4$-vertex graph contains a $4$-cycle.
The following is our main theorem. 

\begin{theorem}\label{thm:main-1mod3}
For $n\geq 3$ and $n\neq 4$,   
\[  
ex_2(n,\mathcal{C}_{1\bmod 3})
= f(n)
:=\begin{cases}
\lceil\frac{3}{2}n \rceil-2, &  \text{if }3\leq n \leq 14 \text{ and }n\not\in\{4,5,10,13\};\\
\lceil\frac{3}{2}n \rceil-3, &  \text{if } n=5;\\
\lceil\frac{3}{2}n \rceil, &  \text{if }n=10;\\
\lceil\frac{3}{2}n \rceil-1, &  \text{if } n\geq 13 \text{ and } n\neq 14.
\end{cases}
\]
The family $\mathcal{H}_n$ of extremal graphs for $ex_2(n,\mathcal{C}_{1\bmod 3})$ is given in Subsection~\ref{subsec:H_def}.
\end{theorem}

As a consequence of Theorem~\ref{thm:main-1mod3}, we characterize the extremal graphs for $ex(n,\mathcal{C}_{1\bmod 3})$ for all values of $n$, see Subsection~\ref{subsec:allextremal}.


\subsection*{Organization}

The rest of this paper is organized as follows. 
In Section~\ref{sec:extremal-graphs}, we define the extremal graph families $\mathcal{H}_n$ for $ex_2(n,\mathcal{C}_{1\bmod 3})$, establish their basic properties, and also characterize all extremal graphs for $ex(n,\mathcal{C}_{1\bmod 3})$. Section~\ref{sec:lemmas and theorems} collects several results needed in our proofs. The proof of Theorem~\ref{thm:main-1mod3}, concerning $2$-connected graphs without $(1 \bmod 3)$-cycles, is given in Section~\ref{sec:proofs}. Section~\ref{sec:concluding} concludes with determining $ex_2(n,\mathcal{C}_{2\bmod 4})$ and an open problem on $ex_2(n,\mathcal{C}_{2\bmod 2k})$ for $k\ge 3$.

\section{Extremal graph families}
\label{sec:extremal-graphs}

\subsection{Extensions and extendable pairs}

Let $i\in [1,3]$ and $H$ be a graph. For two distinct vertices $x,y\in V(H)$, we call $\{x,y\}$ an $i$-\textit{extendable pair} of $H$ if $H$ has no $(-i\bmod 3)$-paths between $x$ and $y$. If $\{x,y\}$ is an $i$-extendable pair of $H$ and $L$ is the graph obtained from $H$ by adding an $(x,y)$-path of length $i+1$ which is internally vertex-disjoint with $H$, then we call $L$ an $i$-\textit{extension} of $H$. If $H$ has exactly one $i$-extension up to isomorphism, then we will use $L_i(H)$ to denote an $i$-extension of $H$. 
We denote by $\mathcal{L}_i(H)$ the set of $i$-extensions of $H$. For a family $\mathcal{H}$ of graphs, we set $\mathcal{L}_i(\mathcal{H})=\bigcup_{H\in\mathcal{H}}\mathcal{L}_i(H)$. 

A \textit{thread} of a graph is a path whose internal vertices all have degree exactly $2$. A thread of length $k$ is called a \textit{$k$-thread}. We call a thread (resp.\ $k$-thread) from $x$ to $y$ an $(x,y)$-\textit{thread} (resp.\ $(x,y)$-$k$-\textit{thread}). Note that every $2$-extension or $3$-extension of a graph always contains a $3$-thread. 
The length of a walk $W$, denoted by $\ell(W)$, is the number of edges in $W$.

\begin{fact}\label{FaExtensionedge}
    Let $H$ be a graph without $(1\bmod 3)$-cycles, and let $L$ be an $i$-extension of $H$ for some $i\in[1,3]$. Then $L$ has no $(1\bmod 3)$-cycles, $n(L)=n(H)+i$, and $e(L)=e(H)+i+1$.
\end{fact}

\begin{proof}
    Suppose that $\{x,y\}$ is an $i$-extendable pair of $H$, and $P$ is the added $(x,y)$-path in the definition of the $i$-extension $L$ of $H$. So $\ell(P)=i+1$. If $L$ has a $(1\bmod 3)$-cycle $C$, then $C$ contains some edges of $P$. Since $P$ is a thread in $L$, we see that $C$ contains all edges of $P$. Now $C-E(P)$ is an $(x,y)$-path in $H$ of length $(1-\ell(P))\equiv(-i)\pmod 3$, a contradiction. The other assertion can be immediately deduced from the definition of the $i$-extensions.
\end{proof}

\begin{fact}\label{Fa3Extension}
    Let $H$ be a graph without $(1\bmod 3)$-cycles and $x,y\in V(H)$. If $xy\in E(H)$, then $\{x,y\}$ is a $3$-extendable pair of $H$. Furthermore, if $H$ is an extremal graph for $ex_2(n,\mathcal{C}_{1\bmod 3})$, then $\{x,y\}$ is a $3$-extendable pair if and only if $xy\in E(H)$.
\end{fact}

\begin{proof}
    If $H$ has an $(x,y)$-path $P$ of length $0\bmod 3$, then $P$ does not contain the edge $xy$, and $P+xy$ is a $(1\bmod 3)$-cycle of $H$, a contradiction. Thus $H$ has no $(x,y)$-paths of length $0\bmod 3$, and $\{x,y\}$ is a 3-extendable pair of $H$. Suppose $\{x,y\}$ is a 3-extendable pair of $H$ and $xy\notin E(H)$. Let $H'=H+xy$. If $H'$ contains a $(1\bmod 3)$-cycle $C$, then $C$ must contain the edge $xy$, and $C-xy$ is an $(x,y)$-path in $H$ of length $0\bmod 3$, a contradiction. It follows that $H'$ contains no $(1\bmod 3)$-cycles and thus $H$ is not an extremal graph for $ex_2(n,\mathcal{C}_{1\bmod 3})$.
\end{proof}

\begin{fact}\label{Fa2Extension}
    Let $H$ be a graph without $(1\bmod 3)$-cycles and $x,y\in V(H)$. If $H$ has an {$(x,y)$-$3$-thread}, then $\{x,y\}$ is a $2$-extendable pair of $H$. 
\end{fact}

\begin{proof}
    Let $Q$ be an {$(x,y)$-$3$-thread} of $H$. If $H$ has an $(x,y)$-path $P$ of length $1\bmod 3$, then $P$ does not contain any edge of $Q$, and $P\cup Q$ is a $(1\bmod 3)$-cycle of $H$, a contradiction. Thus $H$ has no $(x,y)$-paths of length $1\bmod 3$, and $\{x,y\}$ is a 2-extendable pair of $H$. 
\end{proof}

One can find in Table \ref{table:smalln} the value of $f(n)$ for small $n$.
We finish this subsection with a simple observation on the values of $f(n)$.
\begin{table}[htbp]
\centering 
\begin{tabular}{c||c|c|c|c|c|c|c|c|c|c|c|c}
\hline 
$n$ & 3 & 5 & 6 & 7 & 8 & 9 & 10 & 11 & 12 & 13 & 14 & $n\geq 15$ \\\hline 
$f(n)$ & 3 & 5 & 7 & 9 & 10 & 12 & 15 & 15 & 16 & 19 & 19 &  $\lceil\frac{3}{2}n \rceil-1$ \\\hline
\end{tabular}
\caption{The value of $f(n)$.}
\label{table:smalln}
\end{table} 

\begin{fact}\label{Fafn}
    For every $n\geq 7$, $f(n)\geq f(n-2)+3$ unless $n=12$, and the equality holds if and only if $n\not\in\{7,10,12,13,16\}$. For every $n\geq 8$, $f(n)\geq f(n-3)+4$, and the equality holds if and only if $n=13$ or $n\geq 12$ is even.
\end{fact}

\begin{proof}
    Both assertions can be computed by the formula of $f(n)$ (see Table \ref{table:smalln}).
\end{proof}

\subsection{Definition of $\mathcal{H}_n$}\label{subsec:H_def}
We define a graph class $\mathcal{H}_n$ whose members are $2$-connected graphs on $n$ vertices with $f(n)$ edges and no $\modthree{1}$-cycles. Later, we will show that $\mathcal{H}_n$ consists precisely of the extremal graphs in our main theorem when $n\geq 3$ and $n\neq 4$. 

For a graph $G$, we denote by $G^-$ a graph obtained from $G$ by removing exactly one edge. 
First, we define the graphs $H_n$ for each $n\in\{3\}\cup[5,11]$,
as follows (see Figure \ref{FiGraphHi}):
\begin{itemize}
\item $H_3$ is isomorphic to $C_3$;
\item $H_5$ is isomorphic to $C_5$;
\item $H_6$ is the union of a $C_3$ and a $C_5$ which have an edge in common;
\item $H_7$ is obtained from a $K_4$ by subdividing all edges of a $K_{1,3}$ exactly once;
\item $H_8$ is obtained from a $K_4$ by subdividing all edges of a $C_4$ exactly once;
\item $H_9$ is obtained from a $K_{3,3}$ by subdividing all edges of a $3K_2$ exactly once;
\item $H_{10}$ is the Petersen graph;
\item $H_{11}$ is obtained from a $K_5^-$ by subdividing all edges incident to $x$ or $y$ exactly once, where $x$ and $y$ are nonadjacent vertices in $K_5^-$. 
\end{itemize}

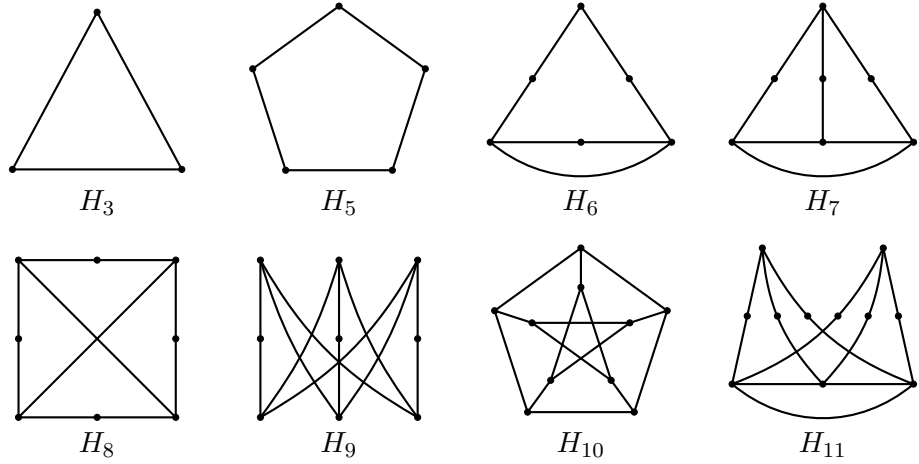
\begin{figure}[htbp]
\centering
\begin{tikzpicture}[scale=0.4]

\begin{scope}[xshift=-12cm]
\draw[fill=black] (0,2.8) {coordinate (x1)} circle (0.1);
\draw[fill=black] (-2.8,-2.4) {coordinate (x2)} circle (0.1);
\draw[fill=black] (2.8,-2.4) {coordinate (x3)} circle (0.1);
\draw[thick] (x1)--(x2)--(x3)--(x1);
\node at (0,-3.5) {$H_3$};
\end{scope}

\begin{scope}[xshift=-4cm]
\foreach \x in {1,2,...,5} \draw[fill=black] (\x*72+18:3) {coordinate (x\x)} circle (0.1);
\draw[thick] (x1)--(x2)--(x3)--(x4)--(x5)--(x1);
\node at (0,-3.5) {$H_5$};
\end{scope}

\begin{scope}[xshift=4cm]
\draw[fill=black] (0,3) {coordinate (x1)} circle (0.1);
\draw[fill=black] (-1.6,0.6) {coordinate (y1)} circle (0.1);
\draw[fill=black] (1.6,0.6) {coordinate (y3)} circle (0.1);
\draw[fill=black] (-3,-1.5) {coordinate (z1)} circle (0.1);
\draw[fill=black] (0,-1.5) {coordinate (z2)} circle (0.1);
\draw[fill=black] (3,-1.5) {coordinate (z3)} circle (0.1);
\draw[thick] (x1)--(z1) (x1)--(z3) (z1)--(z3);
\draw[thick] (z1) edge [bend right=40] (z3);
\node at (0,-3.5) {$H_6$};
\end{scope}

\begin{scope}[xshift=12cm]
\draw[fill=black] (0,3) {coordinate (x1)} circle (0.1);
\draw[fill=black] (-1.6,0.6) {coordinate (y1)} circle (0.1);
\draw[fill=black] (0,0.6) {coordinate (y2)} circle (0.1);
\draw[fill=black] (1.6,0.6) {coordinate (y3)} circle (0.1);
\draw[fill=black] (-3,-1.5) {coordinate (z1)} circle (0.1);
\draw[fill=black] (0,-1.5) {coordinate (z2)} circle (0.1);
\draw[fill=black] (3,-1.5) {coordinate (z3)} circle (0.1);
\draw[thick] (x1)--(z1) (x1)--(z2) (x1)--(z3) (z1)--(z2)--(z3);
\draw[thick] (z1) edge [bend right=40] (z3);
\node at (0,-3.5) {$H_7$};
\end{scope}

\begin{scope}[xshift=-12cm, yshift=-8cm]
\draw[fill=black] (-2.6,-2.6) {coordinate (x1)} circle (0.1);
\draw[fill=black] (-2.6,2.6) {coordinate (x2)} circle (0.1);
\draw[fill=black] (2.6,-2.6) {coordinate (x3)} circle (0.1);
\draw[fill=black] (2.6,2.6) {coordinate (x4)} circle (0.1);
\draw[fill=black] (-2.6,0) {coordinate (y1)} circle (0.1);
\draw[fill=black] (0,-2.6) {coordinate (y2)} circle (0.1);
\draw[fill=black] (2.6,0) {coordinate (y3)} circle (0.1);
\draw[fill=black] (0,2.6) {coordinate (y4)} circle (0.1);
\draw[thick] (x1)--(x2)--(x3)--(x4)--(x1);
\draw[thick] (x1)--(x3) (x2)--(x4);
\node at (0,-3.5) {$H_8$};
\end{scope}

\begin{scope}[xshift=-4cm, yshift=-8cm]
\draw[fill=black] (-2.6,-2.6) {coordinate (x1)} circle (0.1);
\draw[fill=black] (0,-2.6) {coordinate (x2)} circle (0.1);
\draw[fill=black] (2.6,-2.6) {coordinate (x3)} circle (0.1);
\draw[fill=black] (-2.6,0) {coordinate (y1)} circle (0.1);
\draw[fill=black] (0,0) {coordinate (y2)} circle (0.1);
\draw[fill=black] (2.6,0) {coordinate (y3)} circle (0.1);
\draw[fill=black] (-2.6,2.6) {coordinate (z1)} circle (0.1);
\draw[fill=black] (0,2.6) {coordinate (z2)} circle (0.1);
\draw[fill=black] (2.6,2.6) {coordinate (z3)} circle (0.1);
\draw[thick] (x1)--(z1) (x2)--(z2) (x3)--(z3);
\draw[thick] (x1) edge [bend right=15] (z3);
\draw[thick] (x3) edge [bend left=15] (z1);
\draw[thick] (x1) edge [bend right=10] (z2);
\draw[thick] (x2) edge [bend left=10] (z1);
\draw[thick] (x2) edge [bend right=10] (z3);
\draw[thick] (x3) edge [bend left=10] (z2);
\node at (0,-3.5) {$H_9$};
\end{scope}

\begin{scope}[xshift=4cm, yshift=-8cm]
\foreach \x in {1,2,...,5} {\draw[fill=black] (\x*72+18:3) {coordinate (x\x)} circle (0.1);
\draw[fill=black] (\x*72+18:1.7) {coordinate (y\x)} circle (0.1);
\draw[thick] (x\x)--(y\x);}
\draw[thick] (x1)--(x2)--(x3)--(x4)--(x5)--(x1);
\draw[thick] (y1)--(y3)--(y5)--(y2)--(y4)--(y1);
\node at (0,-3.5) {$H_{10}$};
\end{scope}

\begin{scope}[xshift=12cm, yshift=-8cm]
\draw[fill=black] (-2,3) {coordinate (x1)} circle (0.1);
\draw[fill=black] (2,3) {coordinate (x2)} circle (0.1);
\foreach \x in {1,2,...,6} \draw[fill=black] (\x-3.5,0.75) {coordinate (y\x)} circle (0.1);
\draw[fill=black] (-3,-1.5) {coordinate (z1)} circle (0.1);
\draw[fill=black] (0,-1.5) {coordinate (z2)} circle (0.1);
\draw[fill=black] (3,-1.5) {coordinate (z3)} circle (0.1);
\draw[thick] (x1)--(z1) (x2)--(z3);
\draw[thick] (x1) edge [bend right=19] (z2);
\draw[thick] (x2) edge [bend left=19] (z2);
\draw[thick] (x1) edge [bend right=21] (z3);
\draw[thick] (x2) edge [bend left=21] (z1);
\draw[thick] (z1)--(z3);
\draw[thick] (z1) edge [bend right=40] (z3);
\node at (0,-3.5) {$H_{11}$};
\end{scope}

\end{tikzpicture}
\caption{Graphs $H_n$, $n\in\{3\}\cup[5,11]$.}
\label{FiGraphHi}
\end{figure}

Now we define the class $\mathcal{H}_n$ for all $n\geq 3$, $n\neq 4$: 
\begin{itemize}
\item For $n\in\{3,5,6,7,9,10,11\}$, $\mathcal{H}_n=\{H_n\}$;
\item $\mathcal{H}_8=\{H_8,L_2(H_6)\}$ (see Figure~\ref{FiGraphL2H6});
\item $\mathcal{H}_{12}=\mathcal{L}_3(H_9)\cup\mathcal{L}_2(\mathcal{H}'_{10})$, where $\mathcal{H}'_{10}=\mathcal{L}_3(H_7)\cup\mathcal{L}_2(\mathcal{H}_8)$ (see Figures~\ref{FiClassH'10} and \ref{FiGraphH12});
\item $\mathcal{H}_{13}=\{L_3(H_{10})\}$;
\item $\mathcal{H}_{16}=\mathcal{L}_3(\mathcal{H}_{13})$;
\item for even $n\geq 14$ with $n\neq 16$, $\mathcal{H}_n=\mathcal{L}_3(\mathcal{H}_{n-3})\cup\mathcal{L}_2(\mathcal{H}_{n-2})$;
\item for odd $n\geq 15$, $\mathcal{H}_n=\mathcal{L}_2(\mathcal{H}_{n-2})$.
\end{itemize}

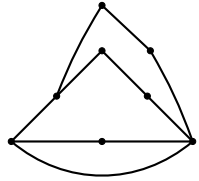
\begin{figure}[htbp]
\centering
\begin{tikzpicture}[scale=0.4]

\draw[fill=black] (0,3) {coordinate (x1)} circle (0.1);
\draw[fill=black] (1.6,1.5) {coordinate (x2)} circle (0.1);
\draw[fill=black] (-1.5,0) {coordinate (y1)} circle (0.1);
\draw[fill=black] (0,1.5) {coordinate (y2)} circle (0.1);
\draw[fill=black] (1.5,0) {coordinate (y3)} circle (0.1);
\draw[fill=black] (-3,-1.5) {coordinate (z1)} circle (0.1);
\draw[fill=black] (0,-1.5) {coordinate (z2)} circle (0.1);
\draw[fill=black] (3,-1.5) {coordinate (z3)} circle (0.1);
\draw[thick] (x1)--(x2) (y2)--(z1) (y2)--(z3) (z1)--(z3);
\draw[thick] (y1) edge [bend left=5] (x1);
\draw[thick] (x2) edge [bend left=5] (z3);
\draw[thick] (z1) edge [bend right=40] (z3);

\end{tikzpicture}
\caption{Graph $L_2(H_6)$.}
\label{FiGraphL2H6}
\end{figure}

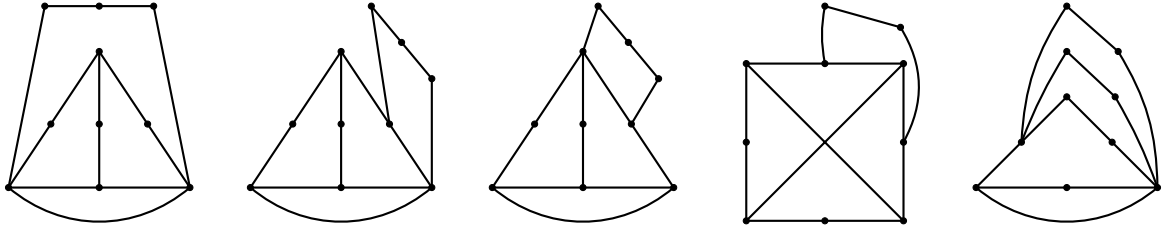
\begin{figure}[htbp]
\centering
\begin{tikzpicture}[scale=0.4]

\begin{scope}[xshift=-16cm]
\draw[fill=black] (0,3) {coordinate (x1)} circle (0.1);
\draw[fill=black] (-1.6,0.6) {coordinate (y1)} circle (0.1);
\draw[fill=black] (0,0.6) {coordinate (y2)} circle (0.1);
\draw[fill=black] (1.6,0.6) {coordinate (y3)} circle (0.1);
\draw[fill=black] (-3,-1.5) {coordinate (z1)} circle (0.1);
\draw[fill=black] (0,-1.5) {coordinate (z2)} circle (0.1);
\draw[fill=black] (3,-1.5) {coordinate (z3)} circle (0.1);
\draw[thick] (x1)--(z1) (x1)--(z2) (x1)--(z3) (z1)--(z2)--(z3);
\draw[thick] (z1) edge [bend right=40] (z3);
\draw[fill=black] (-1.8,4.5) {coordinate (w1)} circle (0.1);
\draw[fill=black] (0,4.5) {coordinate (w2)} circle (0.1);
\draw[fill=black] (1.8,4.5) {coordinate (w3)} circle (0.1);
\draw[thick] (z1)--(w1)--(w3)--(z3); 
\end{scope}

\begin{scope}[xshift=-8cm]
\draw[fill=black] (0,3) {coordinate (x1)} circle (0.1);
\draw[fill=black] (-1.6,0.6) {coordinate (y1)} circle (0.1);
\draw[fill=black] (0,0.6) {coordinate (y2)} circle (0.1);
\draw[fill=black] (1.6,0.6) {coordinate (y3)} circle (0.1);
\draw[fill=black] (-3,-1.5) {coordinate (z1)} circle (0.1);
\draw[fill=black] (0,-1.5) {coordinate (z2)} circle (0.1);
\draw[fill=black] (3,-1.5) {coordinate (z3)} circle (0.1);
\draw[thick] (x1)--(z1) (x1)--(z2) (x1)--(z3) (z1)--(z2)--(z3);
\draw[thick] (z1) edge [bend right=40] (z3);
\draw[fill=black] (1,4.5) {coordinate (w1)} circle (0.1);
\draw[fill=black] (2,3.3) {coordinate (w2)} circle (0.1);
\draw[fill=black] (3,2.1) {coordinate (w3)} circle (0.1);
\draw[thick] (y3)--(w1)--(w3)--(z3); 
\end{scope}

\begin{scope}[xshift=0cm]
\draw[fill=black] (0,3) {coordinate (x1)} circle (0.1);
\draw[fill=black] (-1.6,0.6) {coordinate (y1)} circle (0.1);
\draw[fill=black] (0,0.6) {coordinate (y2)} circle (0.1);
\draw[fill=black] (1.6,0.6) {coordinate (y3)} circle (0.1);
\draw[fill=black] (-3,-1.5) {coordinate (z1)} circle (0.1);
\draw[fill=black] (0,-1.5) {coordinate (z2)} circle (0.1);
\draw[fill=black] (3,-1.5) {coordinate (z3)} circle (0.1);
\draw[thick] (x1)--(z1) (x1)--(z2) (x1)--(z3) (z1)--(z2)--(z3);
\draw[thick] (z1) edge [bend right=40] (z3);
\draw[fill=black] (0.5,4.5) {coordinate (w1)} circle (0.1);
\draw[fill=black] (1.5,3.3) {coordinate (w2)} circle (0.1);
\draw[fill=black] (2.5,2.1) {coordinate (w3)} circle (0.1);
\draw[thick] (x1)--(w1)--(w3)--(y3); 
\end{scope}

\begin{scope}[xshift=8cm]
\draw[fill=black] (-2.6,-2.6) {coordinate (x1)} circle (0.1);
\draw[fill=black] (-2.6,2.6) {coordinate (x2)} circle (0.1);
\draw[fill=black] (2.6,-2.6) {coordinate (x3)} circle (0.1);
\draw[fill=black] (2.6,2.6) {coordinate (x4)} circle (0.1);
\draw[fill=black] (-2.6,0) {coordinate (y1)} circle (0.1);
\draw[fill=black] (0,-2.6) {coordinate (y2)} circle (0.1);
\draw[fill=black] (2.6,0) {coordinate (y3)} circle (0.1);
\draw[fill=black] (0,2.6) {coordinate (y4)} circle (0.1);
\draw[thick] (x1)--(x2)--(x3)--(x4)--(x1);
\draw[thick] (x1)--(x3) (x2)--(x4);
\draw[fill=black] (0,4.5) {coordinate (w1)} circle (0.1);
\draw[fill=black] (2.5,3.8) {coordinate (w2)} circle (0.1);
\draw[thick] (w1)--(w2);
\draw[thick] (y4) edge [bend left=10] (w1);
\draw[thick] (w2) edge [bend left=30] (y3); 
\end{scope}

\begin{scope}[xshift=16cm]
\draw[fill=black] (0,3) {coordinate (x1)} circle (0.1);
\draw[fill=black] (1.6,1.5) {coordinate (x2)} circle (0.1);
\draw[fill=black] (-1.5,0) {coordinate (y1)} circle (0.1);
\draw[fill=black] (0,1.5) {coordinate (y2)} circle (0.1);
\draw[fill=black] (1.5,0) {coordinate (y3)} circle (0.1);
\draw[fill=black] (-3,-1.5) {coordinate (z1)} circle (0.1);
\draw[fill=black] (0,-1.5) {coordinate (z2)} circle (0.1);
\draw[fill=black] (3,-1.5) {coordinate (z3)} circle (0.1);
\draw[thick] (x1)--(x2) (y2)--(z1) (y2)--(z3) (z1)--(z3);
\draw[thick] (y1) edge [bend left=5] (x1);
\draw[thick] (x2) edge [bend left=5] (z3);
\draw[thick] (z1) edge [bend right=40] (z3);
\draw[fill=black] (0,4.5) {coordinate (w1)} circle (0.1);
\draw[fill=black] (1.7,3) {coordinate (w2)} circle (0.1);
\draw[thick] (w1)--(w2);
\draw[thick] (y1) edge [bend left=15] (w1);
\draw[thick] (w2) edge [bend left=15] (z3); 
\end{scope}

\end{tikzpicture}
\caption{All non-isomorphic graphs in the class $\mathcal{H}'_{10}$.}
\label{FiClassH'10}
\end{figure}

\begin{figure}[htbp]
\centering
\begin{tikzpicture}[scale=0.4]

\begin{scope}[xshift=-16cm]
\draw[fill=black] (-2.6,-2.6) {coordinate (x1)} circle (0.1);
\draw[fill=black] (0,-2.6) {coordinate (x2)} circle (0.1);
\draw[fill=black] (2.6,-2.6) {coordinate (x3)} circle (0.1);
\draw[fill=black] (-2.6,0) {coordinate (y1)} circle (0.1);
\draw[fill=black] (0,0) {coordinate (y2)} circle (0.1);
\draw[fill=black] (2.6,0) {coordinate (y3)} circle (0.1);
\draw[fill=black] (-2.6,2.6) {coordinate (z1)} circle (0.1);
\draw[fill=black] (0,2.6) {coordinate (z2)} circle (0.1);
\draw[fill=black] (2.6,2.6) {coordinate (z3)} circle (0.1);
\draw[thick] (x1)--(z1) (x2)--(z2) (x3)--(z3);
\draw[thick] (x1) edge [bend right=15] (z3);
\draw[thick] (x3) edge [bend left=15] (z1);
\draw[thick] (x1) edge [bend right=10] (z2);
\draw[thick] (x2) edge [bend left=10] (z1);
\draw[thick] (x2) edge [bend right=10] (z3);
\draw[thick] (x3) edge [bend left=10] (z2);
\draw[fill=black] (-0.8,6) {coordinate (v1)} circle (0.1);
\draw[fill=black] (1,5) {coordinate (v2)} circle (0.1);
\draw[fill=black] (2.8,4) {coordinate (v3)} circle (0.1);
\draw[thick] (v1)--(v3);
\draw[thick] (z1) edge [bend left=20] (v1);
\draw[thick] (v3) edge [bend left=30] (x3); 
\end{scope}

\begin{scope}[xshift=-8cm]
\draw[fill=black] (-2.6,-2.6) {coordinate (x1)} circle (0.1);
\draw[fill=black] (0,-2.6) {coordinate (x2)} circle (0.1);
\draw[fill=black] (2.6,-2.6) {coordinate (x3)} circle (0.1);
\draw[fill=black] (-2.6,0) {coordinate (y1)} circle (0.1);
\draw[fill=black] (0,0) {coordinate (y2)} circle (0.1);
\draw[fill=black] (2.6,0) {coordinate (y3)} circle (0.1);
\draw[fill=black] (-2.6,2.6) {coordinate (z1)} circle (0.1);
\draw[fill=black] (0,2.6) {coordinate (z2)} circle (0.1);
\draw[fill=black] (2.6,2.6) {coordinate (z3)} circle (0.1);
\draw[thick] (x1)--(z1) (x2)--(z2) (x3)--(z3);
\draw[thick] (x1) edge [bend right=15] (z3);
\draw[thick] (x3) edge [bend left=15] (z1);
\draw[thick] (x1) edge [bend right=10] (z2);
\draw[thick] (x2) edge [bend left=10] (z1);
\draw[thick] (x2) edge [bend right=10] (z3);
\draw[thick] (x3) edge [bend left=10] (z2);
\draw[fill=black] (-0.8,6) {coordinate (v1)} circle (0.1);
\draw[fill=black] (1,5) {coordinate (v2)} circle (0.1);
\draw[fill=black] (2.8,4) {coordinate (v3)} circle (0.1);
\draw[thick] (v1)--(v3);
\draw[thick] (z3) edge [bend left=20] (v1);
\draw[thick] (v3) edge [bend left=20] (y3); 
\end{scope}

\begin{scope}[xshift=0cm]
\draw[fill=black] (0,3) {coordinate (x1)} circle (0.1);
\draw[fill=black] (-1.6,0.6) {coordinate (y1)} circle (0.1);
\draw[fill=black] (0,0.6) {coordinate (y2)} circle (0.1);
\draw[fill=black] (1.6,0.6) {coordinate (y3)} circle (0.1);
\draw[fill=black] (-3,-1.5) {coordinate (z1)} circle (0.1);
\draw[fill=black] (0,-1.5) {coordinate (z2)} circle (0.1);
\draw[fill=black] (3,-1.5) {coordinate (z3)} circle (0.1);
\draw[thick] (x1)--(z1) (x1)--(z2) (x1)--(z3) (z1)--(z2)--(z3);
\draw[thick] (z1) edge [bend right=40] (z3);
\draw[fill=black] (-1.8,4.5) {coordinate (w1)} circle (0.1);
\draw[fill=black] (0,4.5) {coordinate (w2)} circle (0.1);
\draw[fill=black] (1.8,4.5) {coordinate (w3)} circle (0.1);
\draw[thick] (z1)--(w1)--(w3)--(z3);
\draw[fill=black] (-1.5,6) {coordinate (v1)} circle (0.1);
\draw[fill=black] (1.5,6) {coordinate (v2)} circle (0.1);
\draw[thick] (w1)--(v1)--(v2);
\draw[thick] (v2) edge [bend left=25] (z3); 
\end{scope}

\begin{scope}[xshift=8cm]
\draw[fill=black] (0,3) {coordinate (x1)} circle (0.1);
\draw[fill=black] (-1.6,0.6) {coordinate (y1)} circle (0.1);
\draw[fill=black] (0,0.6) {coordinate (y2)} circle (0.1);
\draw[fill=black] (1.6,0.6) {coordinate (y3)} circle (0.1);
\draw[fill=black] (-3,-1.5) {coordinate (z1)} circle (0.1);
\draw[fill=black] (0,-1.5) {coordinate (z2)} circle (0.1);
\draw[fill=black] (3,-1.5) {coordinate (z3)} circle (0.1);
\draw[thick] (x1)--(z1) (x1)--(z2) (x1)--(z3) (z1)--(z2)--(z3);
\draw[thick] (z1) edge [bend right=40] (z3);
\draw[fill=black] (1,4.5) {coordinate (w1)} circle (0.1);
\draw[fill=black] (2,3.3) {coordinate (w2)} circle (0.1);
\draw[fill=black] (3,2.1) {coordinate (w3)} circle (0.1);
\draw[thick] (y3)--(w1)--(w3)--(z3);
\draw[fill=black] (0.7,6) {coordinate (v1)} circle (0.1);
\draw[fill=black] (3,3.9) {coordinate (v2)} circle (0.1);
\draw[thick] (w1)--(v1)--(v2);
\draw[thick] (v2) edge [bend left=25] (z3); 
\end{scope}

\begin{scope}[xshift=16cm]
\draw[fill=black] (0,3) {coordinate (x1)} circle (0.1);
\draw[fill=black] (-1.6,0.6) {coordinate (y1)} circle (0.1);
\draw[fill=black] (0,0.6) {coordinate (y2)} circle (0.1);
\draw[fill=black] (1.6,0.6) {coordinate (y3)} circle (0.1);
\draw[fill=black] (-3,-1.5) {coordinate (z1)} circle (0.1);
\draw[fill=black] (0,-1.5) {coordinate (z2)} circle (0.1);
\draw[fill=black] (3,-1.5) {coordinate (z3)} circle (0.1);
\draw[thick] (x1)--(z1) (x1)--(z2) (x1)--(z3) (z1)--(z2)--(z3);
\draw[thick] (z1) edge [bend right=40] (z3);
\draw[fill=black] (1,4.5) {coordinate (w1)} circle (0.1);
\draw[fill=black] (2,3.3) {coordinate (w2)} circle (0.1);
\draw[fill=black] (3,2.1) {coordinate (w3)} circle (0.1);
\draw[thick] (y3)--(w1)--(w3)--(z3);
\draw[fill=black] (0.7,6) {coordinate (v1)} circle (0.1);
\draw[fill=black] (3,3.9) {coordinate (v2)} circle (0.1);
\draw[thick] (v1)--(v2)--(w3);
\draw[thick] (y3) edge [bend left=15] (v1); 
\end{scope}

\begin{scope}[xshift=-12cm, yshift=-11cm]
\draw[fill=black] (0,3) {coordinate (x1)} circle (0.1);
\draw[fill=black] (-1.6,0.6) {coordinate (y1)} circle (0.1);
\draw[fill=black] (0,0.6) {coordinate (y2)} circle (0.1);
\draw[fill=black] (1.6,0.6) {coordinate (y3)} circle (0.1);
\draw[fill=black] (-3,-1.5) {coordinate (z1)} circle (0.1);
\draw[fill=black] (0,-1.5) {coordinate (z2)} circle (0.1);
\draw[fill=black] (3,-1.5) {coordinate (z3)} circle (0.1);
\draw[thick] (x1)--(z1) (x1)--(z2) (x1)--(z3) (z1)--(z2)--(z3);
\draw[thick] (z1) edge [bend right=40] (z3);
\draw[fill=black] (0.5,4.5) {coordinate (w1)} circle (0.1);
\draw[fill=black] (1.5,3.3) {coordinate (w2)} circle (0.1);
\draw[fill=black] (2.5,2.1) {coordinate (w3)} circle (0.1);
\draw[thick] (x1)--(w1)--(w3)--(y3);
\draw[fill=black] (1,6) {coordinate (v1)} circle (0.1);
\draw[fill=black] (3.4,3.6) {coordinate (v2)} circle (0.1);
\draw[thick] (w1)--(v1)--(v2); 
\draw[thick] (v2) edge [bend left=35] (y3); 
\end{scope}

\begin{scope}[xshift=-4cm, yshift=-11cm]
\draw[fill=black] (0,3) {coordinate (x1)} circle (0.1);
\draw[fill=black] (-1.6,0.6) {coordinate (y1)} circle (0.1);
\draw[fill=black] (0,0.6) {coordinate (y2)} circle (0.1);
\draw[fill=black] (1.6,0.6) {coordinate (y3)} circle (0.1);
\draw[fill=black] (-3,-1.5) {coordinate (z1)} circle (0.1);
\draw[fill=black] (0,-1.5) {coordinate (z2)} circle (0.1);
\draw[fill=black] (3,-1.5) {coordinate (z3)} circle (0.1);
\draw[thick] (x1)--(z1) (x1)--(z2) (x1)--(z3) (z1)--(z2)--(z3);
\draw[thick] (z1) edge [bend right=40] (z3);
\draw[fill=black] (0.5,4.5) {coordinate (w1)} circle (0.1);
\draw[fill=black] (1.5,3.3) {coordinate (w2)} circle (0.1);
\draw[fill=black] (2.5,2.1) {coordinate (w3)} circle (0.1);
\draw[thick] (x1)--(w1)--(w3)--(y3);
\draw[fill=black] (1,6) {coordinate (v1)} circle (0.1);
\draw[fill=black] (3.4,3.6) {coordinate (v2)} circle (0.1);
\draw[thick] (v1)--(v2)--(w3); 
\draw[thick] (x1) edge [bend left=45] (v1); 
\end{scope}

\begin{scope}[xshift=4cm, yshift=-11cm]
\draw[fill=black] (-2.6,-2.6) {coordinate (x1)} circle (0.1);
\draw[fill=black] (-2.6,2.6) {coordinate (x2)} circle (0.1);
\draw[fill=black] (2.6,-2.6) {coordinate (x3)} circle (0.1);
\draw[fill=black] (2.6,2.6) {coordinate (x4)} circle (0.1);
\draw[fill=black] (-2.6,0) {coordinate (y1)} circle (0.1);
\draw[fill=black] (0,-2.6) {coordinate (y2)} circle (0.1);
\draw[fill=black] (2.6,0) {coordinate (y3)} circle (0.1);
\draw[fill=black] (0,2.6) {coordinate (y4)} circle (0.1);
\draw[thick] (x1)--(x2)--(x3)--(x4)--(x1);
\draw[thick] (x1)--(x3) (x2)--(x4);
\draw[fill=black] (0,4.5) {coordinate (w1)} circle (0.1);
\draw[fill=black] (2.5,3.8) {coordinate (w2)} circle (0.1);
\draw[thick] (w1)--(w2);
\draw[thick] (y4) edge [bend left=10] (w1);
\draw[thick] (w2) edge [bend left=30] (y3);
\draw[fill=black] (-0.5,6) {coordinate (v1)} circle (0.1);
\draw[fill=black] (3,5) {coordinate (v2)} circle (0.1);
\draw[thick] (v1)--(v2);
\draw[thick] (y4) edge [bend left=35] (v1);
\draw[thick] (v2) edge [bend left=35] (y3); 
\end{scope}

\begin{scope}[xshift=12cm, yshift=-11cm]
\draw[fill=black] (0,3) {coordinate (x1)} circle (0.1);
\draw[fill=black] (1.6,1.5) {coordinate (x2)} circle (0.1);
\draw[fill=black] (-1.5,0) {coordinate (y1)} circle (0.1);
\draw[fill=black] (0,1.5) {coordinate (y2)} circle (0.1);
\draw[fill=black] (1.5,0) {coordinate (y3)} circle (0.1);
\draw[fill=black] (-3,-1.5) {coordinate (z1)} circle (0.1);
\draw[fill=black] (0,-1.5) {coordinate (z2)} circle (0.1);
\draw[fill=black] (3,-1.5) {coordinate (z3)} circle (0.1);
\draw[thick] (x1)--(x2) (y2)--(z1) (y2)--(z3) (z1)--(z3);
\draw[thick] (y1) edge [bend left=5] (x1);
\draw[thick] (x2) edge [bend left=5] (z3);
\draw[thick] (z1) edge [bend right=40] (z3);
\draw[fill=black] (0,4.5) {coordinate (w1)} circle (0.1);
\draw[fill=black] (1.7,3) {coordinate (w2)} circle (0.1);
\draw[thick] (w1)--(w2);
\draw[thick] (y1) edge [bend left=15] (w1);
\draw[thick] (w2) edge [bend left=15] (z3);
\draw[fill=black] (0,6) {coordinate (v1)} circle (0.1);
\draw[fill=black] (1.8,4.5) {coordinate (v2)} circle (0.1);
\draw[thick] (v1)--(v2);
\draw[thick] (y1) edge [bend left=25] (v1);
\draw[thick] (v2) edge [bend left=25] (z3); 
\end{scope}

\end{tikzpicture}
\caption{All non-isomorphic graphs in the class $\mathcal{H}_{12}$.}
\label{FiGraphH12}
\end{figure}
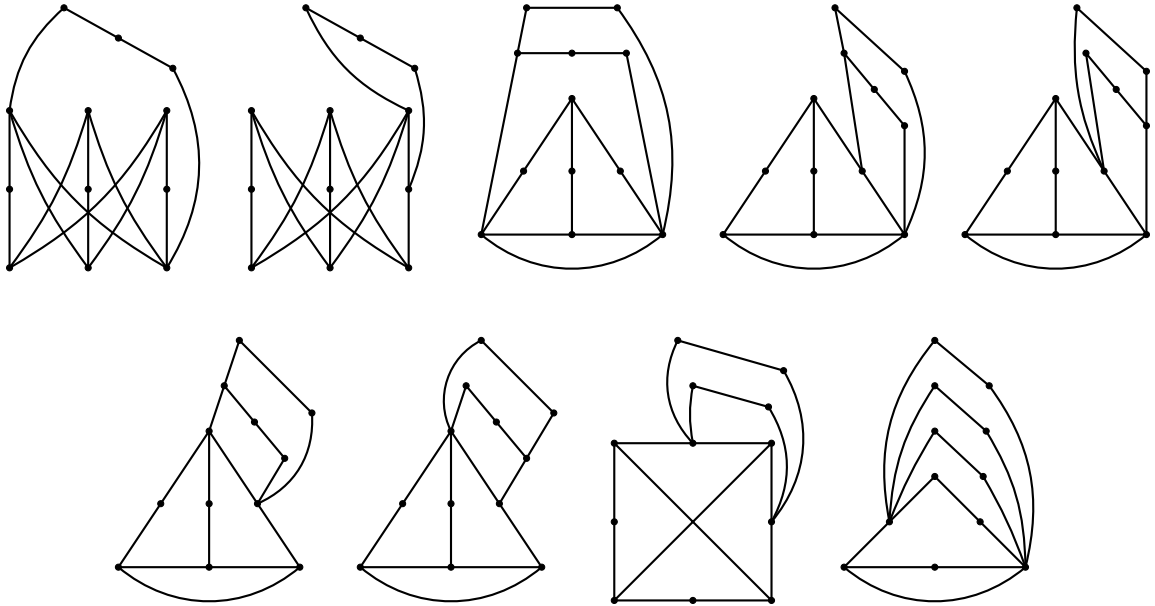

Recall that $H_{10}$ is the Petersen graph. We denote by $H_{10}^-$ the graph obtained from $H_{10}$ by removing an arbitrary edge.

\begin{fact}\label{Fa2Extension2}
    If $H\in\mathcal{H}_3\cup\bigcup_{n=5}^{15}\mathcal{H}_n\cup\mathcal{H}'_{10}\cup\{H_{10}^-\}$, then $\{x,y\}$ is a $2$-extendable pair if and only if $H$ has {an $(x,y)$-$3$-thread}, unless $H\cong H_8$ and $x,y$ have distance $2$ in $H$ and both have degree $2$.
\end{fact}

\begin{proof}
    The assertion can be proved by checking the graphs in $\mathcal{H}_3\cup\bigcup_{n=5}^{15}\mathcal{H}_n\cup\mathcal{H}'_{10}\cup\{H_{10}^-\}$ (see Figures \ref{FiGraphHi}, \ref{FiGraphL2H6}, and \ref{FiClassH'10}). We omit the details here.
\end{proof}

\begin{fact}\label{FaHnNoEmptyeH}
    For $n\geq 3$ with $n\neq 4$, 
    we have $\mathcal{H}_n\neq\emptyset$. Moreover, every graph $H\in\mathcal{H}_n$ is $2$-connected, contains no $(1\bmod 3)$-cycles, and satisfies $e(H)=f(n)$. 
\end{fact}

\begin{proof}
    Recall that every 2-extension or 3-extension of a graph always contains a $3$-thread. Therefore, by Fact \ref{Fa2Extension}, it contains a 2-extendable pair. By Facts~\ref{Fa3Extension} and \ref{Fa2Extension}, we see that $\mathcal{H}_n\neq\emptyset$ for $n\geq 3$ with $n\neq 4$. Now, let $H\in\mathcal{H}_n$. If $H=H_i$ for some $i\in[3,11]\backslash\{4\}$, then one can check that $H$ is $2$-connected, contains no $(1\bmod 3)$-cycles, and satisfies $e(H)=f(n)$. For other cases, $H$ is a 2-extension or 3-extension of some graphs without $(1\bmod 3)$-cycles. Then the assertion can be deduced by Fact~\ref{FaExtensionedge}.
\end{proof}

\begin{fact}\label{FaUniqueHngeq13odd}
    For odd $n\geq 13$, $\mathcal{H}_n$ contains exactly one graph (up to isomorphism). The graph is obtained from a Petersen graph $P$ with an edge $xy$, by adding a new vertex $z$, the edge $xz$, and $k$ internally vertex-disjoint $(y,z)$-$3$-threads whose internal vertices are disjoint from $P$, where $k=\frac{1}{2}(n-11)$ (see Figure \ref{FiGraphHnoddgeq13}).
\end{fact}

\begin{figure}[htbp]
\centering
\begin{tikzpicture}[scale=0.4]

\foreach \x in {1,2,...,5} {\draw[fill=black] (\x*72+18:3) {coordinate (u\x)} circle (0.1);
\draw[fill=black] (\x*72+18:1.7) {coordinate (v\x)} circle (0.1);
\draw[thick] (u\x)--(v\x);}
\draw[thick] (u1)--(u2)--(u3)--(u4)--(u5)--(u1);
\draw[thick] (v1)--(v3)--(v5)--(v2)--(v4)--(v1);
\draw[fill=black] (2.85,-2.43) {coordinate (z)} circle (0.1);
\draw[thick] (z)--(u4);
\draw[fill=black] (3.85,-1.59) {coordinate (z1)} circle (0.1);
\draw[fill=black] (3.85,0.09) {coordinate (y1)} circle (0.1);
\draw[thick] (u5)--(y1)--(z1)--(z);
\draw[fill=black] (5.2,-1.59) {coordinate (z2)} circle (0.1);
\draw[fill=black] (5.2,0.09) {coordinate (y2)} circle (0.1);
\draw[thick] (u5)--(y2)--(z2)--(z);
\draw[fill=black] (7,-1.59) {coordinate (zk)} circle (0.1);
\draw[fill=black] (7,0.09) {coordinate (yk)} circle (0.1);
\draw[thick] (u5)--(yk)--(zk)--(z);
\foreach \x in {1,2,...,4} \draw[fill=black] (0.35*\x+5.2,-0.75) circle (0.05);
\node[below] at (u4) {\small $x$}; \node[below] at (z) {\small $z$}; \node[above] at (u5) {\small $y$};
\node at (4.4,0) {\small $y_1$}; \node at (5.75,0) {\small $y_2$}; \node at (7.6,0) {\small $y_k$};
\node at (4.4,-1.5) {\small $z_1$}; \node at (5.75,-1.5) {\small $z_2$}; \node at (7.6,-1.5) {\small $z_k$};

\end{tikzpicture}
\caption{The unique graph in $\mathcal{H}_n$, $n\geq 13$ is odd and $k=\frac{1}{2}(n-11)$. }\label{FiGraphHnoddgeq13}
\end{figure}
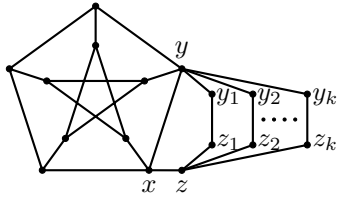

\begin{proof}
    We denote by $\varPi_k$ the graph obtained from a Petersen graph $P$ with an edge $xy$ by adding a new vertex $z$, the edge $xz$, and $k$ internally vertex-disjoint $(y,z)$-$3$-threads whose internal vertices are disjoint from $P$.
    Let $H$ be a graph in $\mathcal{H}_n$. We will show that $H\cong\varPi_k$. If $n=13$, then $H$ is a 3-extension of the Petersen graph $H_{10}$. Note that $H_{10}$ is edge-transitive. By Fact \ref{Fa3Extension}, $H_{10}$ has exactly one 3-extension, and $H=L_3(H_{10})\cong\varPi_1$. If $n=15$, then $H$ is a 2-extension of $\varPi_1$. We label the vertices of $\varPi_1$ as in Figure \ref{FiGraphHnoddgeq13}. One can check that $\varPi_1$ has two 2-extendable pairs, namely $\{x,y_1\}$ and $\{y,z\}$. Note that the 2-extensions of $\varPi_1$ at $\{x,y_1\}$ and $\{y,z\}$ are isomorphic. So $H\cong L_2(\varPi_1)=\varPi_2$. 
    
    Now, we claim that for $k\geq 2$, $\varPi_k$ has exactly one 2-extendable pair, namely $\{y,z\}$. The case $k=2$ can be checked directly; see also Fact~\ref{Fa2Extension}. For $k\geq 3$, let $\{u,v\}$ be a $2$-extendable pair of $\varPi_k$. Removing a $(y,z)$-$3$-thread whose internal vertices contain neither $u$ nor $v$, we obtain a copy of $\varPi_{k-1}$ in which $\{u,v\}$ is still $2$-extendable. 
    Hence, by induction, $\{u,v\}=\{y,z\}$. 
    This proves the claim. 
\end{proof}

\subsection{Extremal graphs for $ex(n,\mathcal{C}_{1\bmod 3})$}\label{subsec:allextremal}
 
By Theorem~\ref{thm:1mod3}, when $9 \mid (n-1)$, the extremal graphs for $ex(n,\mathcal{C}_{1\bmod 3})$ are precisely the graphs whose blocks are isomorphic to the Petersen graph $H_{10}$. However, a complete characterization for the remaining cases $9 \nmid (n-1)$ was not previously known. In this subsection, using the graphs in $\mathcal{H}_n$ from the previous section together with Theorem~\ref{thm:main-1mod3}, we establish a complete characterization of the extremal graphs for $ex(n,\mathcal{C}_{1\bmod 3})$ that holds for all $n$, regardless of whether $9$ divides $n-1$. 
We remark that in \cite{LiPaSh}, all extremal graphs for $\operatorname{ex}(n,\mathcal{C}_{2\bmod 4})$ are characterized, even when $4 \nmid (n-1)$.

\begin{theorem}
A graph $G$ is extremal for $ex(n,\mathcal{C}_{1\bmod 3})$ if and only if $G$ is connected and consists of blocks $B_1,B_2,\ldots,B_k$ with $n(B_i)=n_i$ such that:
\begin{itemize}
    \item[$(1)$] $\sum_{i=1}^kn_i=n+k-1$;
    \item[$(2)$] there exists $t\in[0,4]$ such that $n_i=10$ for all $i\in[t+1,k]$, and\\
    \indent $(a)$ if $t=1$, then $n_1\in\{2,3,6,7,8,9,13,15,17\}$;\\
    \indent $(b)$ if $t=2$, then either $n_1\in\{2,3,6,7,13,15\}$, $n_2=3$, or $(n_1,n_2)=(2,7)$;\\
    \indent $(c)$ if $t=3$, then $(n_1,n_2,n_3)=(2,3,3)$ or $(3,3,3)$ or $(13,3,3)$;\\
    \indent $(d)$ if $t=4$, then $(n_1,n_2,n_3,n_4)=(2,3,3,3)$ or $(3,3,3,3)$;
    \item[$(3)$] $B_i\cong K_2$ if $n_i=2$, and $B_i\in\mathcal{H}_{n_i}$ if $n_i\geq 3$.
\end{itemize}
\end{theorem}

\begin{proof} 
For simplicity, we define $f_1(n)=ex(n,\mathcal{C}_{1\bmod 3})$ for $n\geq 2$, and 
$$f_2(n)=\left\{\begin{array}{ll}
         1, & n=2;  \\
         ex_2(n,\mathcal{C}_{1\bmod 3}), & n=3\mbox{ or }n\geq 5. 
\end{array}\right.$$
Note that $f_2(n)\le f_1(n)$ and $f_1(n)+f_1(10)=f_1(n+9)$ for all $n\geq 2$.
Table \ref{Taf1f2} shows the values of $f_1(n)$ and $f_2(n)$ (see Theorem~\ref{thm:1mod3}) for $n\leq 17$.

\begin{table}[htbp]
\centering 
\begin{tabular}{c||c|c|c|c|c|c|c|c|c|c|c|c|c|c|c|c}
\hline 
$n$      & 2 & 3 & 4 & 5 & 6 & 7 & 8 & 9 & 10 & 11 & 12 & 13 & 14 & 15 & 16 & 17 \\\hline
$f_1(n)$ & $\underline{1}$ & $\underline{3}$ & 4 & 6 & $\underline{7}$ & $\underline{9}$ & $\underline{10}$ & $\underline{12}$ & $\underline{15}$ & 16 & 18 & $\underline{19}$ & 21 & $\underline{22}$ & 24 & $\underline{25}$ \\\hline
$f_2(n)$ & $\underline{1}$ & $\underline{3}$ & - & 5 & $\underline{7}$ & $\underline{9}$ & $\underline{10}$ & $\underline{12}$ & $\underline{15}$ & 15 & 16 & $\underline{19}$ & 19 & $\underline{22}$ & 23 & $\underline{25}$ \\\hline
\end{tabular}
\caption{The values of $f_1(n)$ and $f_2(n)$ for $n\leq 17$, with equal values underlined.}
\label{Taf1f2}
\end{table} 

We first prove the following two facts.

    \begin{fact}\label{Faf1n1nk}
        Let $n_1,n_2,\ldots,n_k$ be positive integers. If $\sum_{i=1}^kf_1(n_i)=f_1(\sum_{i=1}^kn_i-k+1)$, then for any subset $I\subseteq[k]$, it holds $\sum_{i\in I}f_1(n_i)=f_1(\sum_{i\in I}n_i-|I|+1)$.
    \end{fact}

    \begin{proof}
        Let $F_i$ be an extremal graph for $ex(n_i,\mathcal{C}_{1\bmod 3})$, $i\in[k]$, with $u_i\in V(F_i)$. Note that $n(F_i)=n_i$ and $e(F_i)=f_1(n_i)$. Let $F$ be a graph obtained from $\bigcup_{i\in I}F_i$ by identifying all $u_i$ to a common vertex. It follows that $n(F)=\sum_{i\in I}n_i-|I|+1$ and $e(F)=\sum_{i\in I}f_1(n_i)$. Thus we have $\sum_{i\in I}f_1(n_i)\leq f_1(\sum_{i\in I}n_i-|I|+1)$. Let $F'$ be an extremal graph for $ex(\sum_{i\in I}n_i-|I|+1,\mathcal{C}_{1\bmod 3})$ with $u'\in V(F')$, and let $F''$ be a graph obtained from $F'\cup\bigcup_{i\in[k]\backslash I}F_i$ by identifying $u'$ and all $u_i$ to a common vertex. It follows that $n(F'')=\sum_{i=1}^kn_i-k+1$ and $e(F'')=f_1(\sum_{i\in I}n_i-|I|+1)+\sum_{i\in[k]\backslash I}f_1(n_i)$. Thus we have $f_1(\sum_{i\in I}n_i-|I|+1)+\sum_{i\in[k]\backslash I}f_1(n_i)\leq f_1(\sum_{i=1}^kn_i-k+1)$. This implies that $f_1(\sum_{i\in I}n_i-|I|+1)\leq\sum_{i\in I}f_1(n_i)$.
    \end{proof}

    \begin{fact}\label{FaAnn1nt}
        Set $A=\{2,3,6,7,8,9,10,13,15,17\}$. We have
        \begin{itemize}
            \item[$(1)$] $f_1(n)=f_2(n)$ if and only if $n\in A$;
            \item[$(2)$] if $n_1,n_2\in A\backslash\{10\}$, then $f_1(n_1)+f_1(n_2)=f_1(n_1+n_2-1)$ if and only if (up to symmetry) $n_1\in\{2,3,6,7,13,15\}$, $n_2=3$, or $(n_1,n_2)=(2,7)$;
            \item[$(3)$] if $n_1,n_2,n_3\in A\backslash\{10\}$, then $\sum_{i=1}^3f_1(n_i)=f_1(\sum_{i=1}^3n_i-2)$ if and only if (up to symmetry)  $(n_1,n_2,n_3)=(2,3,3)$ or $(3,3,3)$ or $(13,3,3)$;
            \item[$(4)$] if $n_1,n_2,n_3,n_4\in A\backslash\{10\}$, then $\sum_{i=1}^4f_1(n_i)=f_1(\sum_{i=1}^4n_i-3)$ if and only if (up to symmetry) $(n_1,n_2,n_3,n_4)=(2,3,3,3)$ or $(3,3,3,3)$;
            \item[$(5)$] if $n_1,n_2,n_3,n_4,n_5\in A\backslash\{10\}$, then $\sum_{i=1}^5f_1(n_i)<f_1(\sum_{i=1}^5n_i-4)$.
        \end{itemize}
    \end{fact}

    \begin{proof}
        (1) If $n\geq 18$, then one can compute $f_2(n)=\left\lceil\frac{3}{2}n\right\rceil-1<f_1(n)=15q+\left\lfloor\frac{3}{2}r\right\rfloor$, where $n-1=9q+r,\ 0\leq r\leq 8$. For $n\leq 17$, the assertion can be deduced by Table \ref{Taf1f2}.

    (2) We list the values $f_1(n_1)+f_1(n_2)$ and $f_1(n_1+n_2-1)$ in Table \ref{Taf1n1n2}, by which the assertion can be deduced. 

    (3) Suppose that $\sum_{i=1}^3 f_1(n_i)=f_1(\sum_{i=1}^3 n_i-2)$. By Fact \ref{Faf1n1nk}, $f_1(n_i)+f_1(n_j)=f_1(n_i+n_j-1)$ for $\{i,j\}\in\binom{[3]}{2}$. If $n_1\neq 3$ 
    and $n_2\neq 3$, then, up to symmetry, $(n_1,n_2)=(2,7)$ and $n_3=3$ by (2). But now $f_1(2)+f_1(7)+f_1(3)=13<f_1(10)=15$, a contradiction. So we can conclude that $n_2=n_3=3$. By (2) we have $n_1\in\{2,3,6,7,13,15\}$. One can compute that $f_1(n_1)+2f_1(3)\neq f_1(n_1+4)$ for  $n_1\in\{6,7,15\}$. On the other hand, one can compute that $f_1(2)+2f_1(3)=7=f_1(6)$, $3f_1(3)=9=f_1(7)$, and $f_1(13)+2f_1(3)=25=f_1(17)$.

    (4) Suppose that $\sum_{i=1}^4f_1(n_i)=f_1(\sum_{i=1}^4n_i-3)$. By Fact \ref{Faf1n1nk}, $\sum_{i\in I}f_1(n_i)=f_1(\sum_{i\in I}n_i-|I|+1)$ for $I\subseteq[4]$. By (3) we have, up to symmetry, $n_1=2$, $3$, or $13$ and $n_2=n_3=n_4=3$. On the other hand, one can compute that $f_1(2)+3f_1(3)=10=f_1(8)$, $4f_1(3)=12=f_1(9)$, and $f_1(13)+3f_1(3)=28<f_1(19)=30$.

    (5) Suppose that $\sum_{i=1}^5f_1(n_i)=f_1(\sum_{i=1}^5n_i-4)$. By Fact \ref{Faf1n1nk}, $\sum_{i\in I}f_1(n_i)=f_1(\sum_{i\in I}n_i-|I|+1)$ for $I\subseteq[5]$. By (4) we have, up to symmetry, $n_1=2$ or 3, and $n_2=n_3=n_4=n_5=3$. However, one can compute that $f_1(2)+4f_1(3)=13<f_1(10)=15$ and $5f_1(3)=15<f_1(11)=16$.
    \end{proof}

\begin{table}[htbp]
\centering 
\begin{tabular}{c|c|c|c|c|c|c|c|c|c|c}
\hline 
$n_1 \backslash n_2$ & 2 & 3 & 6 & 7 & 8 & 9 & 10 & 13 & 15 & 17 \\\hline
2 & 2,\ 3 & \underline{4,\ 4} & 8,\ 9 & \underline{10,\ 10} & 11,\ 12 & 13,\ 15 & \underline{16,\ 16} & 20,\ 21 & 23,\ 24 & 26,\ 27 \\\hline
3 & & \underline{6,\ 6} & \underline{10,\ 10} & \underline{12,\ 12} & 13,\ 15 & 15,\ 16 & \underline{18,\ 18} & \underline{22,\ 22} & \underline{25,\ 25} & 28,\ 30 \\\hline
6 & & & 14,\ 16 & 16,\ 18 & 17,\ 19 & 19,\ 21 & \underline{22,\ 22} & 26,\ 27 & 29,\ 31 & 32,\ 34 \\\hline
7 & & & & 18,\ 19 & 19,\ 21 & 21,\ 22 & \underline{24,\ 24} & 28,\ 30 & 31,\ 33 & 34,\ 36 \\\hline
8 & & & & & 20,\ 22 & 22,\ 24 & \underline{25,\ 25} & 29,\ 31 & 32,\ 34 & 35,\ 37 \\\hline
9 & & & & & & 24,\ 25 & \underline{27,\ 27} & 31,\ 33 & 34,\ 36 & 37,\ 39 \\\hline
 10 & & & & & & & \underline{30,\ 30} & \underline{34,\ 34} & \underline{37,\ 37} & \underline{40,\ 40} \\\hline
 13 & & & & & & & & 38,\ 39 & 41,\ 42 & 44,\ 46 \\\hline
 15 & & & & & & & & & 44,\ 46 & 47,\ 49 \\\hline
 17 & & & & & & & & & & 50,\ 52 \\\hline
\end{tabular}
\caption{The values of $f_1(n_1)+f_1(n_2)$ and $f_1(n_1+n_2-1)$ for $n_1,n_2 \in A$, with equal values underlined.} 
\label{Taf1n1n2}
\end{table}

Now let $G$ be an extremal graph for $ex(n,\mathcal{C}_{1\bmod 3})$. We have that $G$ is connected, for otherwise adding an edge between different components of $G$ does not create new cycles. Let $B_1,B_2,\ldots,B_k$ be the blocks of $G$, with $n_i=n(B_i)$. So $\sum_{i=1}^kn_i-k+1=n$ and the assertion (1) holds. 

If $n_i=2$, then clearly $B_i\cong K_2$. We have $n_i\neq 4$ for otherwise $B_i$ contains a 4-cycle. If $n_i\geq 3$, then we have $f_1(n_i)=f_2(n_i)$ and $B_i\in\mathcal{H}_{n_i}$; otherwise we can use an extremal graph for $ex(n_i,\mathcal{C}_{1\bmod 3})$ instead of $B_i$ to obtain a graph without $(1\bmod 3)$-cycles and with more edges than $G$. It follows that the assertion (3) holds; and by Fact \ref{FaAnn1nt}, $n_i\in A$ for all $i\in[k]$.

Since $G$ is extremal for $ex(n,\mathcal{C}_{1\bmod 3})$, we have $\sum_{i=1}^kf_1(n_i)=f_1(n)$. Let $I$ be an arbitrary nonempty subset of $[k]$. By Fact \ref{Faf1n1nk}, we see that $\sum_{i\in I}f_1(n_i)=f_1(\sum_{i\in I}n_i-|I|+1)$. 
By Fact \ref{FaAnn1nt}, all but at most four $n_i$'s are equal to $10$. Rearranging the blocks of $G$, we can assume that $n_i\neq 10$ for $i\in[1,t]$ and $n_i=10$ for $i\in[t+1,k]$, where $t\in[0,4]$. Now the assertion (2) can be deduced by Fact \ref{Faf1n1nk}. 
\end{proof}

\section{Some results needed in our proofs}
\label{sec:lemmas and theorems}

To find a $(1\bmod 3)$-cycle, a common and effective approach is to seek two internally vertex-disjoint paths between given vertices and then analyze their lengths to obtain a cycle satisfying the desired length condition. The lemmas below will play an important role in our proofs.

We denote by $N_2(G)$ the set of vertices with degree $2$ in $G$. 

\begin{lemma}[Bai {\it et al.} \cite{B+202503}]\label{lemma: 2 paths mod 3}
Let $G$ be a graph on at least $4$ vertices and $x,y$ two distinct vertices of $G$. If
\begin{itemize}
    \item[$(1)$] $G+xy$ is $2$-connected;
    \item[$(2)$] $N_2(G)\backslash\{x,y\}$ is an independent set; and
    \item[$(3)$] $G$ contains no $4$-cycles,
\end{itemize}
then $G$ contains two $(x,y)$-paths $P_1,P_2$ with $\ell(P_1)\not\equiv \ell(P_2) \pmod 3$. 
\end{lemma}

A vertex cut $S$ of $G$ is {\em essential} if $G-S$ has at least two nontrivial components. We call a connected graph {\em essentially $k$-connected} if it contains no essential vertex cut of size at most $k-1$. 
The following lemma follows directly from the proof of Theorem 2 in \cite{B+202503}. 

\begin{lemma}\label{lem:H10:H10-}
Let $G$ be an essentially $3$-connected graph without $(1\bmod 3)$-cycles.
If $G$ contains two disjoint cycles, then $G$ is isomorphic to $H_{10}$ or $H_{10}^-$.
\end{lemma}

\begin{proof}
See Claim 5 to Claim 15 of the proof of Theorem 2 in \cite{B+202503}. 
\end{proof}

 For $p\geq 3$ and $i\in[0,3]$, let $K^{(i)}_{3,p}$ be a graph obtained from $K_{3,p}$ by adding exactly  $i$ edges between vertices in the partite set of size three. We remark that $K_{3,p}^{(0)}=K_{3,p}$, $K_{3,p}^{(1)}=(K_1\cup K_2)\vee pK_1$, $K_{3,p}^{(2)}=P_3\vee pK_1$, and $K_{3,p}^{(3)}=K_3\vee pK_1$. 
Let $W_p$ be the wheel of order $p+1$, which is the graph obtained from a cycle $C_p$ by adding a new vertex $u$ and $p$ edges from $u$ to all vertices of the cycle $C_p$. That is, $W_p=K_1\vee C_p$.
The following theorem characterizes 3-connected graphs without two disjoint cycles.

\begin{theorem}[Lov\'{a}sz \cite{Lov65}]\label{thm: Lovasz}
    Let $G$ be a $3$-connected graph. 
    If $G$ does not contain two disjoint cycles, 
    then $G$ is isomorphic to $K_5^-,K_5,W_p$ or $K^{(i)}_{3,p}$, where $p\geq 3$ and $i\in[0,3]$.
\end{theorem}

For a graph $H$, we use $H^*$ to denote the graph obtained from $H$ by subdividing its every edge exactly once. 
Let $\hat{W}_4$ denote the graph obtained from $W_4$ by subdividing every edge of some $K_{2,3}$ in $W_4$ exactly once, $K'_{3,3}$ the graph obtained from $K_{3,3}$ by subdividing every edge of some $C_4$ in $K_{3,3}$ exactly once, and $\hat{K}_{3,p}^{(i)}$ the graph obtained from $K_{3,p}^{(i)}$ by subdividing every edge of a $K_{3,p}$ exactly once, as shown in Figure \ref{FiSubdivideW4K33}. Notice that $K_{3,p}^*=\hat{K}_{3,p}^{(0)}$. The following lemma is needed.

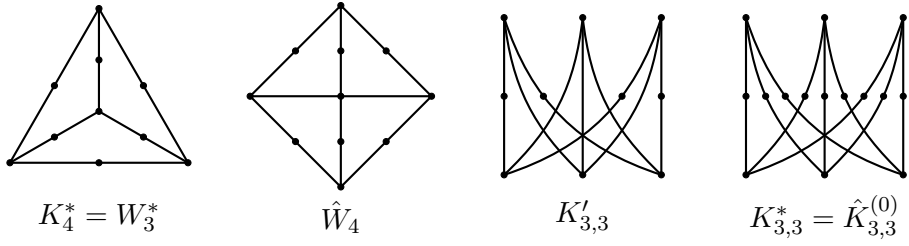
\begin{figure}[htbp]
\centering
\begin{tikzpicture}[scale=0.4]

\begin{scope}[xshift=-12cm, yshift=-0.5cm]
\draw[fill=black] (0,0) {coordinate (o1)} circle (0.1);
\foreach \x in {1,2,3} {\draw[fill=black] (\x*120-30:3.4) {coordinate (x\x)} circle (0.1);
\draw[fill=black] (\x*120-30:1.7) circle (0.1);
\draw[fill=black] (\x*120-90:1.7) circle (0.1);
\draw[thick] (o1)--(x\x);}
\draw[thick] (x1)--(x2)--(x3)--(x1);
\node at (0,-3.5) {$K_4^*=W_3^*$};
\end{scope}

\begin{scope}[xshift=-4cm]
\draw[fill=black] (0,0) {coordinate (o1)} circle (0.1);
\foreach \x in {1,2,...,4} {\draw[fill=black] (\x*90:3) {coordinate (x\x)} circle (0.1);
\draw[fill=black] (\x*90-45:2.12) circle (0.1); \draw[thick] (o1)--(x\x);}
\draw[fill=black] (90:1.5) circle (0.1); \draw[fill=black] (270:1.5) circle (0.1);
\draw[thick] (x1)--(x2)--(x3)--(x4)--(x1);
\node at (0,-4) {$\hat{W}_4$};
\end{scope}

\begin{scope}[xshift=4cm]
\draw[fill=black] (-2.6,-2.6) {coordinate (x1)} circle (0.1);
\draw[fill=black] (0,-2.6) {coordinate (x2)} circle (0.1);
\draw[fill=black] (2.6,-2.6) {coordinate (x3)} circle (0.1);
\foreach \x in {1,3,7,9} \draw[fill=black] (\x*0.65-3.25,0) circle (0.1);
\draw[fill=black] (-2.6,2.6) {coordinate (z1)} circle (0.1);
\draw[fill=black] (0,2.6) {coordinate (z2)} circle (0.1);
\draw[fill=black] (2.6,2.6) {coordinate (z3)} circle (0.1);
\draw[thick] (x1)--(z1) (x2)--(z2) (x3)--(z3);
\draw[thick] (x1) edge [bend right=27] (z3);
\draw[thick] (x3) edge [bend left=27] (z1);
\draw[thick] (x1) edge [bend right=20] (z2);
\draw[thick] (x2) edge [bend left=20] (z1);
\draw[thick] (x2) edge [bend right=20] (z3);
\draw[thick] (x3) edge [bend left=20] (z2);
\node at (0,-4) {$K'_{3,3}$};
\end{scope}

\begin{scope}[xshift=12cm]
\draw[fill=black] (-2.6,-2.6) {coordinate (x1)} circle (0.1);
\draw[fill=black] (0,-2.6) {coordinate (x2)} circle (0.1);
\draw[fill=black] (2.6,-2.6) {coordinate (x3)} circle (0.1);
\foreach \x in {1,2,...,9} \draw[fill=black] (\x*0.65-3.25,0) circle (0.1);
\draw[fill=black] (-2.6,2.6) {coordinate (z1)} circle (0.1);
\draw[fill=black] (0,2.6) {coordinate (z2)} circle (0.1);
\draw[fill=black] (2.6,2.6) {coordinate (z3)} circle (0.1);
\draw[thick] (x1)--(z1) (x2)--(z2) (x3)--(z3);
\draw[thick] (x1) edge [bend right=27] (z3);
\draw[thick] (x3) edge [bend left=27] (z1);
\draw[thick] (x1) edge [bend right=20] (z2);
\draw[thick] (x2) edge [bend left=20] (z1);
\draw[thick] (x2) edge [bend right=20] (z3);
\draw[thick] (x3) edge [bend left=20] (z2);
\node at (0,-4) {$K_{3,3}^*=\hat{K}_{3,3}^{(0)}$};
\end{scope}

\end{tikzpicture}
\caption{Graphs $K_4^*$, $\hat{W}_4$, $K'_{3,3}$ and $\hat{K}_{3,3}^{(0)}$.}
\label{FiSubdivideW4K33}
\end{figure}

\begin{lemma}\label{lemma: no 2 cycles}
Let $G$ be an essentially $3$-connected graph without $(1\bmod 3)$-cycles.
If $G$ does not contain two vertex-disjoint cycles, then $G$ is isomorphic to one of the following graphs,
\begin{equation}\label{members of H}
    H_7,\ H_8,\ H_9,\ H_{11},\ K_4^*,\ \hat{W}_4,\ K'_{3,3}, \, \hat{K}_{3,p}^{(i)} \text{ with } p\geq 3 \text{ and } i\in[0,3].
\end{equation}
\end{lemma}

The proof of Lemma \ref{lemma: no 2 cycles} will be given in Subsection~\ref{sec:lemma no 2 cycles}.

\section{Proofs}\label{sec:proofs}

An {\em inner-vertex} in a block is a vertex that is not a cut-vertex of $G$.
A graph is a {\em block-chain} if it has exactly two end blocks, where an {\em end block} is a block of $G$ containing exactly one cut-vertex of $G$.

\subsection{Proof of Theorem~\ref{thm:main-1mod3}}
\label{sec:proof-1mod3}

By Fact~\ref{FaHnNoEmptyeH},  $ex_2(n,\mathcal{C}_{1\bmod 3})\geq f(n)$ for $n\geq 3$, $n\neq 4$. 
Now let $G$ be a 2-connected $n$-vertex graph without $(1\bmod 3)$-cycles that has maximum number of edges. Then $e(G)\geq f(n)$. We will show that $e(G)=f(n)$ and $G\in\mathcal{H}_n$ by using induction on $n$. 

If $n=3,5$ or $6$, then a graph contains no $(1\bmod 3)$-cycles if and only if it contains no 4-cycles. We have $e(G)\leq ex(n,C_4)=f(n)$. Note that $H_3,H_5,H_6$ are the only $2$-connected extremal graphs for $ex(n,C_4)$ (see \cite{CAJ1989}), and so $G\cong H_3, H_5$ or $H_6$, as desired. 
So assume that $n\geq 7$. Note that we are done if $G\in\mathcal{H}_n$. So assume also that $G\not\in\mathcal{H}_n$ in the following.

\begin{claim}\label{ClAdjacentDegree2}
    $G$ has a $3$-thread.  
\end{claim}

\begin{proof}
    Suppose first that $G$ is essentially 3-connected. If $G$ has two vertex-disjoint cycles, then by Lemma \ref{lem:H10:H10-}, $G\cong H_{10}$ or $H_{10}^-$. However, if $G\cong H_{10}^-$, then $e(G)=14<f(n)$, a contradiction. This contradicts our assumption that $G\notin\mathcal{H}_n$. If $G$ has no two vertex-disjoint cycles, then by Lemma \ref{lemma: no 2 cycles}, $G\in\{H_7,H_8,H_9,H_{11},K_4^*,\hat{W}_4,K_{3,3}'\}\cup\{\hat{K}_{3,p}^{(i)}: p\geq 3, i\in[0,3]\}$. One can check that $e(K_4^*)=12<f(10)$, $e(\hat{W}_4)=14<f(11)$, $e(K_{3,3}')=13<f(10)$ and $e(\hat{K}_{3,p}^{(i)})=6p+i<f(4p+3)=6p+4$ for $p\geq 3$. So we have that $G\cong H_7,H_8,H_9$ or $H_{11}$, as desired. 

    Suppose now that $G$ has an essential 2-cut $\{x,y\}$. Let $Q_1,Q_2$ be two nontrivial components of $G-\{x,y\}$, and let $G_i=G[V(Q_i)\cup\{x,y\}]$, $i=1,2$. Then $G_i+xy$ is 2-connected. If $G$ has no $3$-thread, then $N_2(G_i)\backslash\{x,y\}$ is an independent set. By Lemma \ref{lemma: 2 paths mod 3}, $G_i$ has two $(x,y)$-paths of different lengths modulo 3 for $i=1,2$. It follows that $G$ has three cycles of different lengths modulo 3, one of which is a $(1\bmod 3)$-cycle, a contradiction.
\end{proof}

By Claim \ref{ClAdjacentDegree2}, let $xx'y'y$ be a $3$-thread of $G$. Let $G'=G-\{x',y'\}$. Then $n(G')=n-2$ and $e(G')=e(G)-3\geq f(n)-3$. 

\begin{claim}\label{ClG'2connected}
    $G'$ is $2$-connected.
\end{claim}

\begin{proof}
    Suppose that $G'$ is not 2-connected. Since $G$ is 2-connected, we see that $G'$ is a block-chain such that $x,y$ are inner-vertices of different end-blocks of $G'$. Let $B_1$, $B_2$ be the two end-blocks of $G'$ such that $x,y$ are the inner-vertices of $B_1, B_2$, respectively. 
    
    We claim that $x,y$ have exactly one common neighbor in $G$ (and then in $G'$). If $x,y$ have two common neighbors in $G$, then $G$ has a $4$-cycle, a contradiction. Suppose that $x,y$ have no common neighbors. Let $G''$ be the graph obtained from $G'$ by identifying $x,y$ and denote the new vertex by $z$. Then $G''$ is 2-connected, $n(G'')=n-3$ and $e(G'')=e(G)-3\geq f(n)-3>f(n-3)$ (by Fact \ref{Fafn}). It follows that $G''$ contains a $(1\bmod 3)$-cycle $C$. Recall that $G'$ contains no $(1\bmod 3)$-cycles. The cycle $C$ must pass through $z$ and the predecessor and successor of $z$ are neighbors of $x$ and $y$, respectively, in $G'$. It follows that $G'$ has an $(x,y)$-path of length $1\bmod 3$. Together with $xx'y'y$, we find a $(1\bmod 3)$-cycle of $G$, a contradiction. Thus we conclude that $x,y$ have exactly one common neighbor, say $w$, in $G'$. It follows that $w$ is a cut-vertex of $G'$ and $G'$ has exactly two blocks $B_1,B_2$.

    If one block of $G'$ has only two vertices, say $V(B_2)=\{y,w\}$, then $G''=G'-\{y\}$ is 2-connected without $(1\bmod 3)$-cycles. Note that $n(G'')=n-3$ and $e(G'')=e(G)-4\geq f(n)-4\geq f(n-3)$. By the induction hypothesis, $e(G'')=f(n-3)$, $G''\in\mathcal{H}_{n-3}$ and $G$ is a 3-extension of $G''$. By Fact \ref{Fafn}, we see that $n=13$ or $n\geq 12$ is even, and $G\in\mathcal{H}_n$.

    Now we suppose that both $B_1,B_2$ have at least three vertices. Let $G''$ be the graph obtained from $G'$ by identifying $x,y$ (and identifying the two edges $xw,yw$). Let $z$ be the vertex of $G''$ obtained from identifying $x,y$. Then $G''$ is 2-connected, $n(G'')=n-3$ and $e(G'')=e(G)-4\geq f(n)-4\geq f(n-3)$. If $G''$ contains a $(1\bmod 3)$-cycle, then the cycle passes through $z$ and $G'$ contains an $(x,y)$-path of length $1\bmod 3$. It follows that $G$ contains a $(1\bmod 3)$-cycle, a contradiction. Thus we have that $G''$ contains no $(1\bmod 3)$-cycles. It follows that $G''\in\mathcal{H}_{n-3}$ and $n=13$ or $n\geq 12$ is even by Fact \ref{Fafn}.
    
    Note that $\{z,w\}$ is a clique-$2$-cut of $G''$ (that is, a 2-cut which is also a clique of $G''$). If $n=12$, then $G''\cong H_9$; if $n=13$, then $G''\cong H_{10}$; if $n=14$, then $G''\cong H_{11}$. For all three cases, $G''$ has no clique-$2$-cuts, a contradiction. So assume that $n\geq 16$ is even, and thus $n(G'')\geq 13$ is odd. By Fact \ref{FaUniqueHngeq13odd}, $G''$ is the unique graph in $\mathcal{H}_{n-3}$ (see Figure \ref{FiGraphHnoddgeq13}). Note that $\{z,w\}$ is the only clique-$2$-cut of $G''$. It follows that (up to symmetry) $B_1$ is a Petersen graph and $B_2$ is the graph consisting of one $2$-thread and $k$ $3$-threads between two fixed vertices, where $k=\frac{1}{2}(n-14)$. If $n=16$, then $B_2$ is a 5-cycle. Let $G_1=G-(V(B_2)\backslash\{y,w\})$. Then $G_1\in\mathcal{H}_{13}$ and $G$ is a 3-extension of $G_1$. For $n\geq 18$, let $x_1,y_1$ be two internal vertices of a $3$-thread in $B_2$. Then $G_1=G-\{x_1,y_1\}$ is a 2-connected graph without $(1\bmod 3)$-cycles and $e(G_1)=e(G)-3$. It follows that $G_1\in\mathcal{H}_{n-2}$ and $G$ is a 2-extension of $G_1$. In both cases, we have $G\in\mathcal{H}_n$, as desired.    
\end{proof}

By Fact \ref{Fafn}, $e(G')=e(G)-3\geq f(n)-3\geq f(n-2)$ if $n\neq 12$. By Claim \ref{ClG'2connected}, $e(G')\leq f(n-2)$. It follows that $n\notin\{7,10,12,13,16\}$, $G'\in\mathcal{H}_{n-2}$ and $G$ is a $2$-extension of $G'$. That is, $G\in\mathcal{H}_n$. Now we consider the exceptional case $n=12$. 

\begin{claim}\label{ClPropertyH'10}
    Let $H$ be a $2$-connected graph of order $10$ without $(1\bmod 3)$-cycles. If $e(H)\geq 13$ and $H$ has a $2$-extendable pair, then $e(H)=13$ and $H\in\mathcal{H}'_{10}$.
\end{claim}

\begin{proof}
    We first prove the following subclaims.

    \begin{subclaim}\label{ClH3Thread}
        $H$ has a $3$-thread.
    \end{subclaim}

    \begin{proof}
        Suppose first that $H$ is essentially $3$-connected. If $H$ has two vertex-disjoint cycles, then by Lemma \ref{lem:H10:H10-}, $H\cong H_{10}$ or $H_{10}^-$. However, both $H_{10}$ and $H_{10}^-$ have no $2$-extendable pairs (see Fact \ref{Fa2Extension2}), a contradiction. If $H$ has no two vertex-disjoint cycles, then by Lemma \ref{lemma: no 2 cycles}, $H\in\{H_7,H_8,H_9,H_{11},K_4^*,\hat{W}_4,K_{3,3}'\}\cup\{\hat{K}_{3,p}^{(i)}: p\geq 3, i\in[0,3]\}$. 
        Since $n(H)=10$ and $e(H)\ge 13$, $H \cong K_{3,3}'$.
        However, $K_{3,3}'$ has no $2$-extendable pairs, a contradiction.
        
        Suppose now that $H$ has an essential $2$-cut $\{x,y\}$. Let $Q_1,Q_2$ be two nontrivial components of $H-\{x,y\}$, and let $H_i=H[V(Q_i)\cup\{x,y\}]$, $i=1,2$. Then $H_i+xy$ is $2$-connected. If $H$ has no $3$-thread, then $N_2(H_i)\backslash\{x,y\}$ is an independent set. By Lemma \ref{lemma: 2 paths mod 3}, $H_i$ has two $(x,y)$-paths of different lengths modulo $3$ for $i=1,2$. It follows that $H$ has a $(1\bmod 3)$-cycle, a contradiction. Thus we conclude that $H$ has a $3$-thread.
    \end{proof}

    By Subclaim \ref{ClH3Thread}, let $xx'y'y$ be a $3$-thread of $H$ and $H'=H-\{x',y'\}$. Then $n(H')=8$ and $e(H')=e(H)-3\geq 10$.

    \begin{subclaim}\label{ClH'2connected}
        $H'$ is $2$-connected.
    \end{subclaim}

    \begin{proof}
        Suppose that $H'$ is not 2-connected. Since $H$ is 2-connected, we see that $H'$ is a block-chain such that $x,y$ are inner-vertices of different end-blocks of $H'$. Let $B_1$, $B_2$ be the two blocks of $H'$ such that $x,y$ are the inner-vertices of $B_1, B_2$, respectively. 
    
        We claim that $x,y$ have exactly one common neighbor in $H$ (and then in $H'$). First $x,y$ have no two common neighbors since $H$ contains no 4-cycles. Suppose that $x,y$ have no common neighbors. Let $H''$ be the graph obtained from $H'$ by identifying $x,y$. Then $H''$ is 2-connected, $n(H'')=7$ and $e(H'')=e(H)-3\geq 10>f(7)$. It follows that $H''$ contains a $(1\bmod 3)$-cycle, and thus $H'$ has an $(x,y)$-path of length $1\bmod 3$. Together with $xx'y'y$, we find a $(1\bmod 3)$-cycle of $H$, a contradiction. Thus we conclude that $x,y$ have exactly one common neighbor, say $w$, in $H'$. It follows that $w$ is a cut-vertex of $H'$ and $H'$ has exactly two blocks $B_1,B_2$.

        If one block of $H'$ has only two vertices, say $V(B_2)=\{y,w\}$, then $H''=H'-\{y\}$ is 2-connected without $(1\bmod 3)$-cycles. Note that $n(H'')=7$ and $e(H'')=e(H)-4\geq 9$. By the induction hypothesis, $H''=H_7$ and $H$ is a 3-extension of $H''$. That is, $H\in\mathcal{H}'_{10}$.

        Now we suppose that both $B_1,B_2$ have at least three vertices. Let $H''$ be the graph obtained from $H'$ by identifying $x,y$ (and identifying the two edges $xw,yw$). Let $z$ be the vertex of $H''$ obtained from identifying $x,y$. Then $H''$ is 2-connected, $n(H'')=7$ and $e(H'')=e(H)-4\geq 9$. If $H''$ contains a $(1\bmod 3)$-cycle, then so does $H$, a contradiction. Thus we have that $H''$ contains no $(1\bmod 3)$-cycles. It follows that $H''\cong H_7$. Note that $\{z,w\}$ is a clique-$2$-cut of $H''$. However, $H_7$ has no clique-$2$-cuts, a contradiction.    
    \end{proof}

    By Subclaim~\ref{ClH'2connected}, $H'$ is $2$-connected and without $(1\bmod 3)$-cycles. Note that $n(H')=8$ and $e(H')\geq 10$. By induction hypothesis, $H'\in\mathcal{H}_8$ and $H$ is a $2$-extension of $H'$. That is, $H\in\mathcal{H}'_{10}$.    
\end{proof}

For $n=12$, we have that $n(G')=10$, $e(G')\geq 13$. By Claim \ref{ClG'2connected}, $G'$ is 2-connected without $(1\bmod 3)$-cycles. By Claim~\ref{ClPropertyH'10}, $G'\in\mathcal{H}'_{10}$, and $G$ is a 2-extension of $G'$. That is, $G\in\mathcal{H}_{12}$.
This completes the proof.

\subsection{Proof of Lemma \ref{lemma: no 2 cycles}}
\label{sec:lemma no 2 cycles}

A \textit{simple} subdivision of a graph $H$ is obtained from $H$ by subdividing each edge of $H$ at most once. Suppose $G$ is a simple subdivision of $H$. In what follows, when no confusion arises, we simply say that an edge is subdivided to mean that an edge of $H$ is subdivided in $G$. For a subgraph $H'$ of $H$, we use $\tau(H')$ to denote the number of edges in $H'$ that are subdivided. If $G'$ is the subgraph of $G$ obtained by subdividing each edge of $H'$ at most once, then we say that $H'$ \textit{converts} to $G'$. 

We notice that if $C$ is a triangle in $H$ and $G$ is a simple subdivision of $H$ without $(1\bmod 3)$-cycles, then $\tau(C)\neq 1$; and if $C$ is a 4-cycle, then $\tau(C)\neq 0$ or 3. We first prove the following fact.

\begin{fact}\label{FaSimpleK4K33}
    $(1)$ Every simple subdivision of $K_4$, except $H_7$, $H_8$ and $K_4^*$, contains a $(1\bmod 3)$-cycle.   
    $(2)$ Every simple subdivision of $K_{3,3}$, except $H_9$, $K'_{3,3}$ and $K_{3,3}^*$, contains a $(1\bmod 3)$-cycle.
\end{fact}

\begin{proof}
    (1) Let $H\cong K_4$ and $G$ be a simple subdivision of $H$ without $(1\bmod 3)$-cycles. Set $V(H)=\{v_1,v_2,v_3,v_4\}$. Let $C$ be a triangle of $H$ with $\tau:=\tau(C)$ as small as possible, say $C=v_1v_2v_3v_1$. Note that $\tau\neq 1$ since $G$ contains no $(1\bmod 3)$-cycles.

    Suppose first that $\tau=0$. If $\tau(v_1v_4v_2)=0$, then $v_1v_4v_2v_3v_1$ converts to a 4-cycle; if $\tau(v_1v_4v_2)=1$, then $v_1v_4v_2v_1$ converts to a 4-cycle. Both case we have a contradiction. It follows that $\tau(v_1v_4v_2)=2$, and both two edges $v_1v_4,v_2v_4$ are subdivided. By a similar analysis, we see that $v_3v_4$ is subdivided as well. Thus $G\cong H_7$, as desired.

    Suppose second that $\tau=2$, say $v_1v_2$, $v_2v_3$ are subdivided. By the choice of $C$, we have that $v_1v_4$ and $v_3v_4$ are subdivided. If $v_2v_4$ is subdivided, then $v_1v_4v_2v_3v_1$ converts to a 7-cycle, a contradiction. Thus $v_2v_4$ is not subdivided and $G\cong H_8$, as desired.

    Suppose third that $\tau=3$. By the choice of $C$, each edge of $H$ is subdivided and $G\cong K_4^*$, as desired.

    (2) Let $H\cong K_{3,3}$ and $G$ be a simple subdivision of $H$ without $(1\bmod 3)$-cycles. Let $X=\{x_1,x_2,x_3\}$ and $Y=\{y_1,y_2,y_3\}$ be the partition sets of $H$. Let $K$ be a claw ($K_{1,3}$) of $H$ with $\tau:=\tau(K)$ as small as possible, say $x_1$ is the center of $K$ (see Figure \ref{FiFaK33}).

\begin{figure}[htbp]
\centering
\begin{tikzpicture}[scale=0.4]

\begin{scope}[xshift=-15cm]
\draw[fill=black] (-2,2.5) {coordinate (x2)} circle (0.1);
\draw[fill=black] (2,2.5) {coordinate (x3)} circle (0.1);
\draw[fill=black] (-3,-1.5) {coordinate (y1)} circle (0.1);
\draw[fill=black] (0,-1.5) {coordinate (y2)} circle (0.1);
\draw[fill=black] (3,-1.5) {coordinate (y3)} circle (0.1);
\draw[fill=black] (0,-3.5) {coordinate (x1)} circle (0.1);
\draw[thick] (x2)--(y1) (x3)--(y3);
\draw[thick] (x1)--(y1) (x1)--(y2) (x1)--(y3);
\node[below] at (x1) {$x_1$}; \node[above] at (x2) {$x_2$}; \node[above] at (x3) {$x_3$};
\node[left] at (y1) {$y_1$}; \node[above] at (y2) {$y_2$}; \node[right] at (y3) {$y_3$};
\draw[thick] (x2) edge [bend right=19] (y2);
\draw[thick] (x3) edge [bend left=19] (y2);
\draw[thick] (x2) edge [bend right=21] (y3);
\draw[thick] (x3) edge [bend left=21] (y1);
\foreach \x in {1,3,4,6} \draw[fill=black] (\x-3.5,0.5) circle (0.1);
\node at (0,-5) {$\tau=0$};
\end{scope}

\begin{scope}[xshift=-5cm]
\draw[fill=black] (-2,2.5) {coordinate (x2)} circle (0.1);
\draw[fill=black] (2,2.5) {coordinate (x3)} circle (0.1);
\draw[fill=black] (-3,-1.5) {coordinate (y1)} circle (0.1);
\draw[fill=black] (0,-1.5) {coordinate (y2)} circle (0.1);
\draw[fill=black] (3,-1.5) {coordinate (y3)} circle (0.1);
\draw[fill=black] (0,-3.5) {coordinate (x1)} circle (0.1);
\draw[thick] (x2)--(y1) (x3)--(y3);
\draw[thick] (x1)--(y1) (x1)--(y2) (x1)--(y3);
\node[below] at (x1) {$x_1$}; \node[above] at (x2) {$x_2$}; \node[above] at (x3) {$x_3$};
\node[left] at (y1) {$y_1$}; \node[above] at (y2) {$y_2$}; \node[right] at (y3) {$y_3$};
\draw[thick] (x2) edge [bend right=19] (y2);
\draw[thick] (x3) edge [bend left=19] (y2);
\draw[thick] (x2) edge [bend right=21] (y3);
\draw[thick] (x3) edge [bend left=21] (y1);
\draw[fill=black] (0,-2.5) circle (0.1);
\foreach \x in {1,6} \draw[fill=black] (\x-3.5,0.5) circle (0.1);
\node at (0,-5) {$\tau=1$};
\end{scope}

\begin{scope}[xshift=5cm]
\draw[fill=black] (-2,2.5) {coordinate (x2)} circle (0.1);
\draw[fill=black] (2,2.5) {coordinate (x3)} circle (0.1);
\draw[fill=black] (-3,-1.5) {coordinate (y1)} circle (0.1);
\draw[fill=black] (0,-1.5) {coordinate (y2)} circle (0.1);
\draw[fill=black] (3,-1.5) {coordinate (y3)} circle (0.1);
\draw[fill=black] (0,-3.5) {coordinate (x1)} circle (0.1);
\draw[thick] (x2)--(y1) (x3)--(y3);
\draw[thick] (x1)--(y1) (x1)--(y2) (x1)--(y3);
\node[below] at (x1) {$x_1$}; \node[above] at (x2) {$x_2$}; \node[above] at (x3) {$x_3$};
\node[left] at (y1) {$y_1$}; \node[above] at (y2) {$y_2$}; \node[right] at (y3) {$y_3$};
\draw[thick] (x2) edge [bend right=19] (y2);
\draw[thick] (x3) edge [bend left=19] (y2);
\draw[thick] (x2) edge [bend right=21] (y3);
\draw[thick] (x3) edge [bend left=21] (y1);
\foreach \x in {1,3} \draw[fill=black] (\x*1.5-3,-2.5) circle (0.1);
\foreach \x in {2,5} \draw[fill=black] (\x-3.5,0.5) circle (0.1);
\node at (0,-5) {$\tau=2$};
\end{scope}

\begin{scope}[xshift=15cm]
\draw[fill=black] (-2,2.5) {coordinate (x2)} circle (0.1);
\draw[fill=black] (2,2.5) {coordinate (x3)} circle (0.1);
\draw[fill=black] (-3,-1.5) {coordinate (y1)} circle (0.1);
\draw[fill=black] (0,-1.5) {coordinate (y2)} circle (0.1);
\draw[fill=black] (3,-1.5) {coordinate (y3)} circle (0.1);
\draw[fill=black] (0,-3.5) {coordinate (x1)} circle (0.1);
\draw[thick] (x2)--(y1) (x3)--(y3);
\draw[thick] (x1)--(y1) (x1)--(y2) (x1)--(y3);
\node[below] at (x1) {$x_1$}; \node[above] at (x2) {$x_2$}; \node[above] at (x3) {$x_3$};
\node[left] at (y1) {$y_1$}; \node[above] at (y2) {$y_2$}; \node[right] at (y3) {$y_3$};
\draw[thick] (x2) edge [bend right=19] (y2);
\draw[thick] (x3) edge [bend left=19] (y2);
\draw[thick] (x2) edge [bend right=21] (y3);
\draw[thick] (x3) edge [bend left=21] (y1);
\foreach \x in {1,2,3} \draw[fill=black] (\x*1.5-3,-2.5) circle (0.1);
\foreach \x in {1,2,...,6} \draw[fill=black] (\x-3.5,0.5) circle (0.1);
\node at (0,-5) {$\tau=3$};
\end{scope}

\end{tikzpicture}
\caption{Constructions of $G$ in Fact \ref{FaSimpleK4K33} (2)}\label{FiFaK33}
\end{figure}
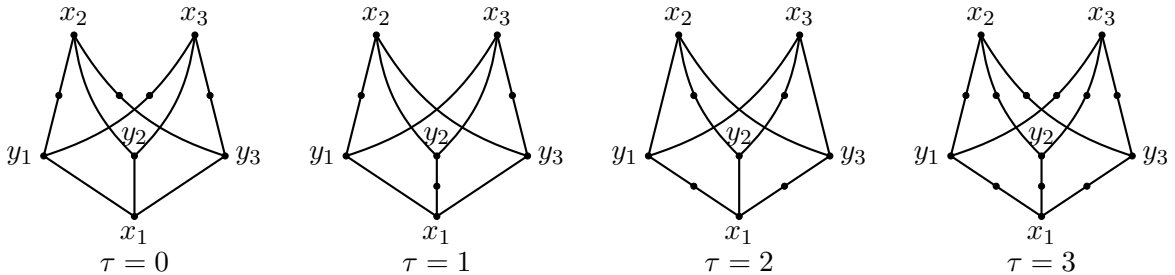

    Suppose first that $\tau=0$. If there are two edges incident to $x_2$ that are not subdivided, say $x_2y_1,x_2y_2$, then $x_1y_1x_2y_2x_1$ converts to a 4-cycle, a contradiction. It follows that at least two edges incident to $x_2$ are subdivided, say $x_2y_1,x_2y_3$. If $\tau(y_1x_3y_3)=0$, then $x_1y_1x_3y_3x_1$ converts to a 4-cycle. If $\tau(y_1x_3y_3)=1$, then $x_2y_1x_3y_3x_2$ converts to a 7-cycle. Both cases we have a contradiction. It follows that $\tau(y_1x_3y_3)=2$, and both two edges $x_3y_1,x_3y_3$ are subdivided. If $x_2y_2$ is subdivided, then $x_1y_3x_3y_1x_2y_2x_1$ converts to a 10-cycle, a contradiction. Thus we conclude that $x_2y_2$, and similarly, $x_3y_2$, is not subdivided. Thus $G\cong K'_{3,3}$, as desired. 

    Suppose second that $\tau=1$, say $x_1y_2$ is subdivided. If $x_2y_2$ is subdivided, then $x_2y_1$ and $x_2y_3$ are not subdivided; for otherwise $x_1y_1x_2y_2x_1$ or $x_1y_2x_2y_3x_1$ converts to a 7-cycle. But now $x_1y_1x_2y_3x_1$ converts to a 4-cycle, a contradiction. Thus we conclude that $x_2y_2$, and similarly, $x_3y_2$, is not subdivided. If $\tau(y_1x_2y_3)=0$, then $x_1y_1x_2y_3x_1$ converts to a 4-cycle, a contradiction. If $\tau(y_1x_2y_3)=2$, then $y_1x_1y_3$, $y_1x_2y_3$, $y_1x_1y_2x_2y_3$ convert to three $(y_1,y_3)$-paths of lengths 2, 4, 6. All the three paths are internally vertex-disjoint with the path that $y_1x_3y_3$ converts to. Thus $G$ contains a $(1\bmod 3)$-cycle, a contradiction. Now we conclude that $\tau(y_1x_2y_3)=1$, say $x_2y_1$ is subdivided. If $x_3y_1$ is subdivided, then $x_1y_1x_3y_2x_2y_3x_1$ converts to a 7-cycle, a contradiction. Thus we conclude that $x_3y_1$ is not subdivided. By the choice of $K$, $x_3y_3$ is subdivided and $G\cong H_9$, as desired.

    Suppose third that $\tau=2$, say $x_1y_1, x_1y_3$ are subdivided. By the choice of $K$, $x_2y_2$ and $x_3y_2$ are subdivided. If $x_2y_1$ is subdivided, then $x_1y_1x_2y_2x_1$ converts to a 7-cycle, a contradiction. Thus we conclude that $x_2y_1$ is not subdivided, and similarly, $x_2y_3, x_3y_1, x_3y_3$ are not subdivided. But now $x_2y_1x_3y_3x_2$ converts to a 4-cycle, a contradiction.

    Suppose fourth that $\tau=3$. By the choice of $K$, each edge of $H$ is subdivided and $G\cong K_{3,3}^*$, as desired.
\end{proof}

Now let $G$ be an essentially $3$-connected graph without $(1\bmod 3)$-cycles such that $G$ does not contain two vertex-disjoint cycles. Let $H$ be the graph obtained from $G$ by contracting one of the two edges incident to each vertex in $N_2(G)$. Then $H$ is 3-connected without two vertex-disjoint cycles, and $G$ is a simple subdivision of $H$. By Theorem \ref{thm: Lovasz}, $H$ is isomorphic to $K_5^-$, $K_5$, $W_{p}$ or $K^{(i)}_{3,p}$, where $p\geq 3$ and $0\leq i\leq 3$. 

\setcounter{case}{0}
\begin{case}\label{CaK5-}
    $H\cong K_5^-$.
\end{case}

Set $V(H)=\{v_1,\ldots,v_5\}$, and let $v_1v_2\notin E(H)$. Let $\tau=\tau(v_3v_4v_5v_3)$ (see Figure \ref{FiCaK5-}). Note that $\tau\neq 1$ and that $H-v_1\cong H-v_2\cong K_4$. By Fact \ref{FaSimpleK4K33}, both $H-v_1$ and $H-v_2$ convert to an $H_7$, $H_8$ or $K_4^*$.

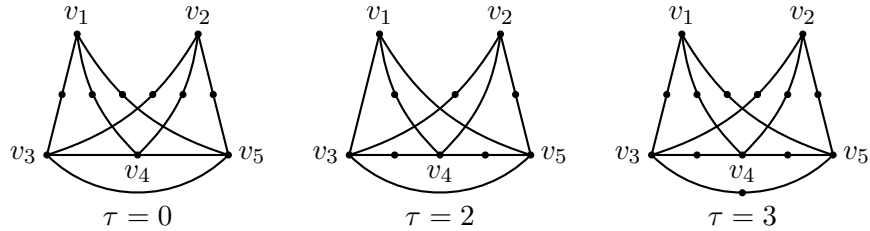
\begin{figure}[htbp]
\centering
\begin{tikzpicture}[scale=0.4]

\begin{scope}[xshift=-10cm]
\draw[fill=black] (-2,2.5) {coordinate (v1)} circle (0.1);
\draw[fill=black] (2,2.5) {coordinate (v2)} circle (0.1);
\draw[fill=black] (-3,-1.5) {coordinate (v3)} circle (0.1);
\draw[fill=black] (0,-1.5) {coordinate (v4)} circle (0.1);
\draw[fill=black] (3,-1.5) {coordinate (v5)} circle (0.1);
\draw[thick] (v1)--(v3) (v2)--(v5);
\draw[thick] (v3)--(v4)--(v5); \draw[thick] (v3) edge [bend right=45] (v5);
\node[above] at (v1) {$v_1$}; \node[above] at (v2) {$v_2$};
\node[left] at (v3) {$v_3$}; \node[below] at (v4) {$v_4$}; \node[right] at (v5) {$v_5$};
\draw[thick] (v1) edge [bend right=19] (v4);
\draw[thick] (v2) edge [bend left=19] (v4);
\draw[thick] (v1) edge [bend right=21] (v5);
\draw[thick] (v2) edge [bend left=21] (v3);
\foreach \x in {1,2,...,6} \draw[fill=black] (\x-3.5,0.5) circle (0.1);
\node at (0,-3.5) {$\tau=0$};
\end{scope}

\begin{scope}
\draw[fill=black] (-2,2.5) {coordinate (v1)} circle (0.1);
\draw[fill=black] (2,2.5) {coordinate (v2)} circle (0.1);
\draw[fill=black] (-3,-1.5) {coordinate (v3)} circle (0.1);
\draw[fill=black] (0,-1.5) {coordinate (v4)} circle (0.1);
\draw[fill=black] (3,-1.5) {coordinate (v5)} circle (0.1);
\draw[thick] (v1)--(v3) (v2)--(v5);
\draw[thick] (v3)--(v4)--(v5); \draw[thick] (v3) edge [bend right=45] (v5);
\node[above] at (v1) {$v_1$}; \node[above] at (v2) {$v_2$};
\node[left] at (v3) {$v_3$}; \node[below] at (v4) {$v_4$}; \node[right] at (v5) {$v_5$};
\draw[thick] (v1) edge [bend right=19] (v4);
\draw[thick] (v2) edge [bend left=19] (v4);
\draw[thick] (v1) edge [bend right=21] (v5);
\draw[thick] (v2) edge [bend left=21] (v3);
\foreach \x in {1,3} \draw[fill=black] (\x*1.5-3,-1.5) circle (0.1);
\foreach \x in {2,4,6} \draw[fill=black] (\x-3.5,0.5) circle (0.1);
\node at (0,-3.5) {$\tau=2$};
\end{scope}

\begin{scope}[xshift=10cm]
\draw[fill=black] (-2,2.5) {coordinate (v1)} circle (0.1);
\draw[fill=black] (2,2.5) {coordinate (v2)} circle (0.1);
\draw[fill=black] (-3,-1.5) {coordinate (v3)} circle (0.1);
\draw[fill=black] (0,-1.5) {coordinate (v4)} circle (0.1);
\draw[fill=black] (3,-1.5) {coordinate (v5)} circle (0.1);
\draw[thick] (v1)--(v3) (v2)--(v5);
\draw[thick] (v3)--(v4)--(v5); \draw[thick] (v3) edge [bend right=45] (v5);
\node[above] at (v1) {$v_1$}; \node[above] at (v2) {$v_2$};
\node[left] at (v3) {$v_3$}; \node[below] at (v4) {$v_4$}; \node[right] at (v5) {$v_5$};
\draw[thick] (v1) edge [bend right=19] (v4);
\draw[thick] (v2) edge [bend left=19] (v4);
\draw[thick] (v1) edge [bend right=21] (v5);
\draw[thick] (v2) edge [bend left=21] (v3);
\draw[fill=black] (0,-2.75) circle (0.1);
\foreach \x in {1,3} \draw[fill=black] (\x*1.5-3,-1.5) circle (0.1);
\foreach \x in {1,2,...,6} \draw[fill=black] (\x-3.5,0.5) circle (0.1);
\node at (0,-3.5) {$\tau=3$};
\end{scope}

\end{tikzpicture}
\caption{Constructions of $G$ in Case \ref{CaK5-}}\label{FiCaK5-}
\end{figure}

Suppose first that $\tau=0$. Then both $H-v_1$ and $H-v_2$ convert to an $H_7$, implying that all edges between $\{v_1,v_2\}$ and $\{v_3,v_4,v_5\}$ are subdivided. That is, $G\cong H_{11}$, as desired.

Suppose second that $\tau=2$. Then both $H-v_1$ and $H-v_2$ convert to an $H_7$ or $H_8$. In each case $\tau(v_3v_1v_4)=\tau(v_4v_2v_5)=1$. Now $v_3v_1v_4v_2v_5v_3$ converts to a 7-cycle, a contradiction.

Suppose third that $\tau=3$. Then both $H-v_1$ and $H-v_2$ converts to a $K_4^*$. This implies that all edges of $H$ are subdivided. Now $v_3v_1v_4v_2v_5v_1$ converts to a 10-cycle, a contradiction.

\begin{case}\label{CaK5}
    $H\cong K_5$.
\end{case}

For each edge $e$ in $H$, $H-e$ is a $K_5^-$. By the analysis of Case \ref{CaK5-}, we see that $H-e$ converts to an $H_{11}$ for each $e\in E(H)$. Thus there are two edges $e_1,e_2\in E(H)$ such that $e_1$ is subdivided and $e_2$ is not. Now $H-e_1$ and $H-e_2$ converts to two graphs with different edge numbers, one of which is not isomorphic to $H_{11}$, a contradiction.

\begin{case}\label{CaWp}
    $H\cong W_p$ with $p\geq 3$.
\end{case}

If $p=3$, then $H\cong K_4$. By Fact \ref{FaSimpleK4K33}, $G\in\{H_7, H_8, K_4^*\}$. Now we assume that $p\geq 4$. Let $u$ be the center of $H$ and $v_1v_2\ldots v_pv_1$ be the $p$-cycle of $H$. Let $C$ be a triangle in $H$ with $\tau:=\tau(C)$ as small as possible, say $C=uv_1v_pu$ (see Figure \ref{FiCaWp}). Note that $\tau\neq 1$. 

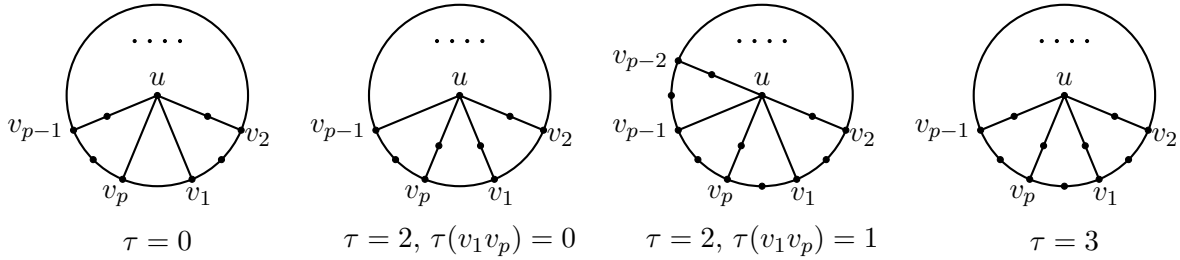
\begin{figure}[htbp]
\centering
\begin{tikzpicture}[scale=0.4]

\begin{scope}[xshift=-15cm]
\draw[thick] (0,0) circle (3); \draw[fill=black] (0,0) {coordinate (u)} circle (0.1);
\foreach \x in {0,1,2,7} {\draw[fill=black] (\x*45-112.5:3) {coordinate (v\x)} circle (0.1);
\draw[thick] (u)--(v\x);
\coordinate (x\x) at (\x*45-112.5:3.6);}
\node[above] at (u) {$u$}; \node[left] at (v7) {$v_{p-1}$}; 
\node at (x1) {$v_1$}; \node at (x2) {$v_2$}; \node at (x0) {$v_p$};
\foreach \x in {1,2,3,4} \draw[fill=black] (\x*0.5-1.25,1.8) circle (0.05);
\foreach \x in {2,7} \draw[fill=black] (\x*45-112.5:1.8) circle (0.1);
\foreach \x in {0,2} \draw[fill=black] (\x*45-135:3) circle (0.1);
\node at (0,-4.8) {$\tau=0$};
\end{scope}

\begin{scope}[xshift=-5cm]
\draw[thick] (0,0) circle (3); \draw[fill=black] (0,0) {coordinate (u)} circle (0.1);
\foreach \x in {0,1,2,7} {\draw[fill=black] (\x*45-112.5:3) {coordinate (v\x)} circle (0.1);
\draw[thick] (u)--(v\x);
\coordinate (x\x) at (\x*45-112.5:3.6);}
\node[above] at (u) {$u$}; \node[left] at (v7) {$v_{p-1}$}; 
\node at (x1) {$v_1$}; \node at (x2) {$v_2$}; \node at (x0) {$v_p$};
\foreach \x in {1,2,3,4} \draw[fill=black] (\x*0.5-1.25,1.8) circle (0.05);
\foreach \x in {0,1,2} \draw[fill=black] (\x*45-112.5:1.8) circle (0.1);
\foreach \x in {0} \draw[fill=black] (\x*45-135:3) circle (0.1);
\node at (0,-4.8) {$\tau=2$, $\tau(v_1v_p)=0$};
\end{scope}

\begin{scope}[xshift=5cm]
\draw[thick] (0,0) circle (3); \draw[fill=black] (0,0) {coordinate (u)} circle (0.1);
\foreach \x in {0,1,2,6,7} {\draw[fill=black] (\x*45-112.5:3) {coordinate (v\x)} circle (0.1);
\draw[thick] (u)--(v\x);
\coordinate (x\x) at (\x*45-112.5:3.6);}
\node[above] at (u) {$u$}; \node[left] at (v6) {$v_{p-2}$}; \node[left] at (v7) {$v_{p-1}$}; 
\node at (x1) {$v_1$}; \node at (x2) {$v_2$}; \node at (x0) {$v_p$};
\foreach \x in {1,2,3,4} \draw[fill=black] (\x*0.5-1.25,1.8) circle (0.05);
\foreach \x in {0,2,6} \draw[fill=black] (\x*45-112.5:1.8) circle (0.1);
\foreach \x in {-1,0,1,2} \draw[fill=black] (\x*45-135:3) circle (0.1);
\node at (0,-4.8) {$\tau=2$, $\tau(v_1v_p)=1$};
\end{scope}

\begin{scope}[xshift=15cm]
\draw[thick] (0,0) circle (3); \draw[fill=black] (0,0) {coordinate (u)} circle (0.1);
\foreach \x in {0,1,2,7} {\draw[fill=black] (\x*45-112.5:3) {coordinate (v\x)} circle (0.1);
\draw[thick] (u)--(v\x);
\coordinate (x\x) at (\x*45-112.5:3.6);}
\node[above] at (u) {$u$}; \node[left] at (v7) {$v_{p-1}$}; 
\node at (x1) {$v_1$}; \node at (x2) {$v_2$}; \node at (x0) {$v_p$};
\foreach \x in {1,2,3,4} \draw[fill=black] (\x*0.5-1.25,1.8) circle (0.05);
\foreach \x in {0,1,2,7} \draw[fill=black] (\x*45-112.5:1.8) circle (0.1);
\foreach \x in {0,1,2} \draw[fill=black] (\x*45-135:3) circle (0.1);
\node at (0,-4.8) {$\tau=3$};
\end{scope}

\end{tikzpicture}
\caption{Constructions of $G$ in Case \ref{CaWp}}\label{FiCaWp}
\end{figure}

Suppose first that $\tau=0$. If $\tau(uv_2v_1)=0$, then $uv_pv_1v_2u$ converts to a 4-cycle; if $\tau(uv_2v_1)=1$, then $uv_1v_2u$ converts to a 4-cycle, both contradictions. So we conclude that $\tau(uv_2v_1)=2$, and similarly, $\tau(uv_{p-1}v_p)=2$. That is, all four edges $uv_2,v_1v_2,uv_{p-1},v_{p-1}v_p$ are subdivided. Now $v_{p-1}uv_2$, $v_{p-1}v_pv_1v_2$, $v_{p-1}v_puv_1v_2$ convert to three $(v_{p-1},v_2)$-paths of length 4, 5, 6, respectively. All the three paths are internally vertex-disjoint with the path that $v_2v_3\ldots v_{p-1}$ converts to. Thus $G$ contains a $(1\bmod 3)$-cycle, a contradiction.

Suppose second that $\tau=2$. If $\tau(v_1v_p)=0$, then $uv_1,uv_p$ are subdivided. In this case $\tau(uv_2v_1)\neq 2$; for otherwise $uv_pv_1v_2u$ converts a 7-cycle. It follows that $\tau(uv_2v_1)=1$, and similarly $\tau(uv_{p-1}v_p)=1$. Now $uv_{p-1}v_pv_1v_2u$ converts a 7-cycle, a contradiction. So we conclude that $\tau(v_1v_p)=1$ (i.e., $v_1v_p$ is subdivided), and assume without loss of generality that $uv_p$ is subdivided. If $\tau(uv_{p-1}v_p)=2$, then $uv_{p-1}v_pv_1u$ converts a 7-cycle, a contradiction. Thus we conclude that $\tau(uv_{p-1}v_p)=1$. If $uv_{p-1}$ is subdivided, then we can get a contradiction by the analysis above. So we conclude that $v_{p-1}v_p$ is subdivided. By the choice of $C$, $uv_{p-2}$, $v_{p-2}v_{p-1}$, $v_1v_2$ and $uv_2$ are subdivided. If $p=4$, then $v_{p-2}=v_2$ and $G\cong\hat{W}_4$, as desired. Now assume that $p\geq 5$. It follows that $v_{p-2}uv_2$, $v_{p-2}v_{p-1}uv_1v_2$, $v_{p-2}v_{p-1}v_pv_1v_2$ convert to three $(v_{p-2},v_2)$-paths of lengths 4, 6 and 8, respectively. All the three paths are internally vertex-disjoint with the path that $v_2v_3\ldots v_{p-2}$ converts to. Thus $G$ contains a $(1\bmod 3)$-cycle, a contradiction.

Suppose third that $\tau=3$. By the choice of $C$, each edge of $H$ is subdivided, and $uv_1v_2v_3v_4u$ converts a 10-cycle, a contradiction.

\begin{case}\label{CaK3pi}
    $H\cong K_{3,p}^{(i)}$ with $p\geq 3$ and $i\in[0,3]$.
\end{case}

We let $X=\{x_1,x_2,x_3\}$ and $Y=\{y_1,y_2,\ldots,y_p\}$ be the partition sets of $H$, where $e(H[X])=i$. 

\begin{subcase}
    $p=3$.
\end{subcase}

If $i=0$, then $H\cong K_{3,3}$. By Fact \ref{FaSimpleK4K33}, $G\in\{H_9, K'_{3,3},K_{3,3}^*\}$. Now assume that $i\geq 1$.  Let $H'=H-E(H[X])$ and $G'$ be the subgraph of $G$ that $H'$ converts to. By Fact \ref{FaSimpleK4K33}, $G'\in\{H_9, K'_{3,3},K_{3,3}^*\}$. If $G'\cong H_9$, then each two vertices in $X$ are connected by a 2-path and a 3-path in $G'$. If $G'\cong K'_{3,3}$, then each two vertices in $X$ are connected by a 2-path and a 3- or 6-path in $G'$. Since there is an edge $e\in E(H[X])$, no matter $e$ is subdivided or not, $G$ has a $(1\bmod 3)$-cycle, a contradiction. So we conclude that $G'\cong K_{3,3}^*$, implying that all edges between $X$ and $Y$ are subdivided. Now each two vertices in $X$ are connected by a 8-path in $G$. If some edge in $E(H[X])$ is subdivided, then $G$ contains a 10-cycle, a contradiction. So we conclude that every edge in $E(H[X])$ is not subdivided, and $G\cong\hat{K}_{3,3}^{(i)}$.

\begin{subcase}
    $p\geq 4$.
\end{subcase}

For simplicity we denote by $H_{ijk}$, $1\leq i<j<k\leq p$, the subgraph of $H$ induced by the edges between $X$ and $\{y_i,y_j,y_k\}$, and let $G_{ijk}$ be the subgraph of $G$ that $H_{ijk}$ converts to. By Fact \ref{FaSimpleK4K33}, we have $G_{ijk}\in\{H_9,K'_{3,3}, K_{3,3}^*\}$. Let $K$ be a claw of $H$ with center in $Y$ with $\tau:=\tau(K)$ as small as possible, say $y_1$ is the center of $K$.

Suppose first that $\tau=0$. Then $G_{1jk}\cong K'_{3,3}$ for $2\leq j<k\leq p$. It follows that each vertex in $Y\backslash\{y_1\}$ is incident to exactly two subdivided edges. Now $G_{234}$ is not isomorphic to $H_9$, $K'_{3,3}$ and $K_{3,3}^*$, a contradiction.

Suppose second that $\tau=1$. Then $G_{1jk}\cong H_9$ for $2\leq j<k\leq p$. It follows that each vertex in $Y\backslash\{y_1\}$ is incident to exactly one subdivided edge. Thus two subdivided edges between $X$ and $Y$ are adjacent, say $x_1y_2,x_1y_3$. Now $G_{234}$ is not isomorphic to $H_9$, $K'_{3,3}$ and $K_{3,3}^*$, a contradiction.

Suppose third that $\tau=2$. Then $G_{1jk}\cong K'_{3,3}$ for $2\leq j<k\leq p$. It follows that some vertex in $Y$ is not incident to any subdivided edge, contradicting the choice of $K$.

Suppose fourth that $\tau=3$. By the choice of $K$, all edges between $X$ and $Y$ are subdivided. Note that each two vertices in $X$ are connected by a 8-path in $G$. This implies that all edges in $E(H[X])$ is not subdivided. Thus $G\cong\hat{K}_{3,p}^{(i)}$, as desired.

\section{Concluding remarks}
\label{sec:concluding}

We conclude with some remarks on graphs without $(2\bmod 4)$-cycles.
Gao {\it et al.} \cite{G+24} showed that if $G$ contains no $(2\bmod 4)$-cycle, then $e(G)\leq \frac{5}{2} (n-1)$ and, for $4\mid (n-1)$, the equality holds if and only if each block of $G$ is isomorphic to $K_5$. 
We consider the 2-connected setting and determine the corresponding maximum number of edges. For $k\geq 0$, we let $F_k=K_2\vee(K_2\cup kK_1)$ (see Figure \ref{FiGraphFk}). Note that $F_0=K_4$. The following results are needed in our proof.

\begin{figure}[htbp]
\centering
\begin{tikzpicture}[scale=0.4]

\foreach \x in {1,2} \draw[fill=black] (\x*2-3,0) {coordinate (x\x)} circle (0.1);
\foreach \y in {1,2} \draw[fill=black] (\y*2-3,2) {coordinate (y\y)} circle (0.1);
\foreach \z in {1,2,4,5} \draw[fill=black] (\z*2-6,-2.5) {coordinate (z\z)} circle (0.1);
\draw[thick] (x1)--(x2) (y1)--(y2);
\foreach \x in {1,2} \foreach \y in {1,2} \draw[thick] (x\x)--(y\y);
\foreach \x in {1,2} \foreach \z in {1,2,4,5} \draw[thick] (x\x)--(z\z);
\foreach \x in {1,2,3,4} \draw[fill=black] (\x*0.6-1.5,-2.5) circle (0.05);
\node[left] at (x1) {$x_1$}; \node[right] at (x2) {$x_2$};
\node[left] at (y1) {$y_1$}; \node[right] at (y2) {$y_2$};
\node[below] at (z1) {$z_1$}; \node[below] at (z2) {$z_2$};
\node[below] at (z4) {$z_{k-1}$}; \node[below] at (z5) {$z_k$};

\end{tikzpicture}
\caption{Graph $F_k$.}
\label{FiGraphFk}
\end{figure}
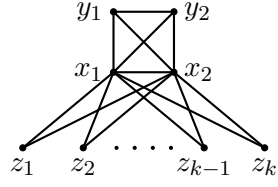

\begin{theorem}[Fan \cite{Fan84}]\label{thm:Fan}
  Let $G$ be a $2$-connected graph on $n\ge 3$ vertices such that for any two vertices of distance two, there is one with degree at least $\frac{n}{2}$. Then $G$ contains a Hamilton cycle.
\end{theorem}

\begin{theorem}[Bai {\it et al.} \cite{B+202511}]\label{thm:previous-2mod4}
    Let $G$ be a $2$-connected graph with minimum degree at least $k-1$ and order at least $k+2$, where $k\ge 4$ is even. Then $G$ contains cycles of lengths $\ell$ modulo $k$ for all even $\ell$.
\end{theorem}

\begin{theorem}\label{thm:main-2mod4}
For $n\geq 3$, 
\[ ex_2(n,\mathcal{C}_{2\bmod 4})
~=~\begin{cases}
    \binom{n}{2}, &  \text{if } n\in[3,5]; \\
    2n-2, & \text{if } n\ge 6.
\end{cases} \]
The extremal graphs are $K_n$ for $n\in[3,5]$ and $F_{n-4}$ for $n\ge 6$. 
\end{theorem}
 
\begin{proof}
We apply induction on $n$. For $n\in[3,5]$, since $K_n$ contains no $(2\bmod 4)$-cycles, the assertion then holds clearly. Assume that $n\geq 6$. Let $G$ be a $2$-connected $n$-vertex graph without $(2\bmod 4)$-cycles and with the maximum number of edges subject to this condition. In view of $F_{n-4}$, we have $e(G)\geq 2n-2$. We will show that $e(G)=2n-2$ and $G\cong F_{n-4}$.
 
For $n=6$, by Theorem \ref{thm:Fan} and $G$ containing no 6-cycle, there exist $u,v$ with distance two in $G$ and $\max\{d_G(u),d_G(v)\}<\frac{n}{2}=3$. By $G$ being 2-connected, $d_G(u)=d_G(v)=2$. It follows that $e(G)=4+e(G-\{u,v\})\leq 4+\binom{4}{2}=10$ and the equality holds exactly when $G-\{u,v\}\cong K_4$. By $e(G)\ge 10$, we have $e(G)=10$. If $N_G(u)\neq N_G(v)$, then $G$ clearly contains a 6-cycle. So $N_G(u)=N_G(v)$ and $G\cong F_2$, as desired. 

Assume that $n\ge 7$. Then there is a vertex $v$ of degree $2$ in $G$ by Theorem~\ref{thm:previous-2mod4}. Let $G'=G-v$. Then $e(G')=e(G)-2\geq 2n-4$. If $G'$ is 2-connected, then by the induction hypothesis, $G'\cong F_{n-5}$. We label the vertices of $G'$ as in Figure \ref{FiGraphFk}. We notice that each two vertices of $G'$, except $x_1,x_2$, are connected by a 4-path. Since $G$ contains no $(2\bmod 4)$-cycles, we have $N_G(v)=\{x_1,x_2\}$, and $G\cong F_{n-4}$, as desired. So we conclude that $G'$ is not 2-connected.

Since $G$ is 2-connected and $d_G(v)=2$, we see that $G'$ is a block-chain, and $N_G(v)$ consists of two inner-vertices of different end-blocks of $G'$. Let $H_1,H_2,\ldots,H_t$ be the blocks of $G'$, and let $x_i$, $i\in[1,t-1]$, be the cut-vertex of $G'$ common to $H_i$ and $H_{i+1}$. Let $x_0,x_t$ be the two neighbors of $v$ in $H_1$ and $H_t$, respectively.

For any $H_i$ isomorphic to neither $K_5$ nor $K_5^-$, one can see that $e(H_i)\leq 2n_i-2$, and the equality holds if and only if $n_i\geq 4$, $H_i=F_{n_i-4}$, or $n_i=5$, $H_i=W_4$. If $H_i$ is isomorphic to neither $K_5$ nor $K_5^-$ for any $i\in[1,t]$, then by the induction hypothesis, 
\[ 
2n-4\le e(G')=\sum_{i=1}^{t}e(H_i)\le\sum_{i=1}^{t}(2n_i-2)=2(n-1+t-1)-2t=2n-4.
\]
Thus $e(G')=2n-4$. It follows that $n_i\ge 4$ and $H_i$ is isomorphic to $W_4$ or $F_k$ for any $i\in[1,t]$. We notice that each two vertices of a $W_4$ or $F_k$ are connected by three paths of consecutive lengths. Thus $G'$ has five $(x_0,x_t)$-paths of consecutive lengths, one of which, together with $x_tvx_0$, form a $(2\bmod 4)$-cycle in $G$, a contradiction. So we can assume that $H_r\cong K_5$ or $K_5^-$ for some $r\in[1,t]$. 

If $H_r\cong K_5$ or $H_r\cong K_5^-$ and $x_{r-1}x_r\in E(H_r)$, then $H_r$ contains four $(x_{r-1},x_r)$-paths of lengths 1, 2, 3, 4, respectively, one of which, together with an $(x_r,x_{r-1})$-path with all internal vertices outside $H_r$, form a $(2\bmod 4)$-cycle in $G$, a contradiction. So we conclude that $H_r\cong K_5^-$ and $x_{r-1}x_r\notin E(H_r)$.

Now we see that $H_r$ contains three $(x_{r-1},x_r)$-paths of lengths 2, 3, 4, respectively. If there exists $i\in[1,t]\backslash\{r\}$ such that $H_i\in\{K_5,K_5^-,W_4\}\cup\{F_k: k\geq 0\}$, then $H_i$ contains three $(x_{i-1},x_i)$-paths of consecutive lengths. It follows that $G'$ contains five $(x_0,x_t)$-paths of consecutive lengths, and $G$ contains a $(2\bmod 4)$-cycle, a contradiction. So we conclude that $H_i\notin\{K_5,K_5^-,W_4\}\cup\{F_k: k\geq 0\}$ for any $i\in[1,t]\backslash\{r\}$. Thus $e(H_i)\leq 2n_i-3$ for any $i\in[1,t]\backslash\{r\}$ by the induction hypothesis. 

Now we have
\[ 
2n-4\le e(G')=\sum_{i=1}^{t}e(H_i)\le 9+\sum_{i\in[1,t]\backslash\{r\}}(2n_i-3)=\sum_{i=1}^{t}(2n_i-3)+2=2n-t-2.
\]
It follows that $t=2$, i.e., $G'$ has exactly two blocks $H_1,H_2$. We can assume without loss of generality that $H_r=H_1$. Since the equality holds in the above inequality, we see that $e(H_2)=2n_2-3$.

If $x_1x_2\in E(H_2)$, then an $(x_0,x_1)$-path in $H_1$ of length 3, together with $x_1x_2vx_0$, form a 6-cycle in $G$, a contradiction. So we conclude that $x_1x_2\notin E(H_2)$. If $H_2+x_1x_2$ contains a $(2\bmod 4)$-cycle, then the cycle passes through $x_1x_2$ and $H_2$ contains an $(x_1,x_2)$-path of lenth $(1\bmod 4)$. Together with $x_2vx_0$ and an $(x_0,x_1)$-path in $H_1$ of length 3, we find a $(2\bmod 4)$-cycle in $G$, a contradiction. So we conclude that $H_2+x_1x_2$ contains no $(2\bmod 4)$-cycles. 

By the induction hypothesis, $n_2\geq 4$ and $H_2+x_1x_2\in\{W_4\}\cup\{F_k: k\geq 0\}$. We notice that each edge of $W_4$ or $F_k$ is contained in a triangle. This implies that $H_2$ contains an $(x_1,x_2)$-path of length 2. Together with $x_2vx_0$ and an $(x_0,x_1)$-path in $H_1$ of length 2, we find a 6-cycle in $G$, a contradiction. 
\end{proof}

Finally, we propose the following conjecture for the general cases.
\begin{conjecture}
For $k\geq 3$ and $n$ large enough, 
    \[
        ex_2(n,\mathcal{C}_{2\bmod 2k})=\binom{k}{2}+k(n-k)+1.
    \] 
The extremal graph is $K_k\vee (K_2\cup (n-k-2)K_1)$.
\end{conjecture}

\section*{Acknowledgement}

The research of Yandong Bai and Binlong Li was supported by the National Key Research and Development Program of China (No. 2026YFE0151700), National Natural Science Foundation of China (Nos. 12542043, 12242111, 12131013), and Guangdong Basic and Applied Basic Research Foundation (No. 2023A1515030208). 
Hojin Chu was supported by a KIAS Individual Grant (CG101801) at Korea Institute for Advanced Study. Boram Park was supported by the National Research Foundation of Korea (NRF) grant funded by the Korea government (No. RS-2025-00523206), and supported by the New Faculty Startup Fund from Seoul National University.

\end{spacing}
\end{document}